\documentclass[11pt]{amsart}

\usepackage{amsmath,latexsym,amssymb,amsthm,amsfonts,graphicx}
\usepackage{hyperref,doi}
\usepackage{tikz-cd}
\usepackage{enumitem}
\usepackage{tabularx}
\usepackage{parskip}
\usepackage[margin=1.25in]{geometry}
\usepackage{mathtools}

\setenumerate[1]{label=(\thesection.\arabic*)}

\numberwithin{equation}{section}
\numberwithin{figure}{section}

\newtheorem{theorem}{Theorem}[section]
\newtheorem{lemma}[theorem]{Lemma}

\newtheorem{corollary}[theorem]{Corollary}
\newtheorem{proposition}[theorem]{Proposition}

\theoremstyle{definition}
\newtheorem{remark}[theorem]{Remark}
\newtheorem{remarks}[theorem]{Remarks}

\newtheorem{examples}[theorem]{Examples}

\def\Z{\ensuremath{\mathbb{Z}}}

\def\R{\ensuremath{\mathbb{R}}}

\newcommand{\pa}[1]{\left(#1\right)}
\newcommand{\cpa}[1]{\left\{#1\right\}}
\newcommand{\wt}[1]{\widetilde{#1}}
\newcommand{\tn}[1]{\textnormal{#1}}
\newcommand{\br}[1]{\left[#1\right]}

\newcommand{\Int}[1]{\tn{Int} \, #1}

\newcommand{\card}[1]{\left| #1 \right|}

\font\cuf=cmtt8
\newcommand{\curl}[1]{{\cuf #1}}

\begin{document}
\title{Structure and classification of torus theta-curves}

\author[J.~Calcut]{Jack S. Calcut}
\address{Department of Mathematics\\
         Oberlin College\\
         Oberlin, OH 44074}
\email{jcalcut@oberlin.edu}
\urladdr{\href{https://www2.oberlin.edu/faculty/jcalcut/}{\curl{https://www2.oberlin.edu/faculty/jcalcut/}}}

\author[S.~Nieman]{Samantha E. Nieman}
\address{Department of Mathematics, Statistics, and Computer Science\\
         University of Illinois Chicago\\
         Chicago, Illinois 60607}
\email{sniem@uic.edu}

\makeatletter
\@namedef{subjclassname@2020}{%
  \textup{2020} Mathematics Subject Classification}
\makeatother

\date{\today}

\begin{abstract}
We study theta-curves embedded in a standard torus in the 3-sphere.
We show that each nontrivial torus knot together with an essential arc determines a prime theta-curve,
yielding explicit infinite families of prime theta-curves.
We compute their constituent knots and identify the structure governing these embeddings,
which leads to a complete classification of torus theta-curves up to ambient isotopy and homeomorphism of the 3-sphere.
In particular, Kinoshita’s theta-curve does not lie on a standard torus.
\end{abstract}

\maketitle

\section{Introduction}

A natural way to construct a theta-curve is to adjoin an arc to a knot.
We study this construction in the setting of torus knots and arcs on a standard torus in the $3$-sphere.
Our first main result shows that each nontrivial torus knot together with an essential arc determines a prime theta-curve.
This produces explicit infinite families of prime theta-curves.

Classical knot theory is the placement problem for a circle
in Euclidean $3$-space or in the $3$-sphere.
That problem has been generalized in numerous ways.
Knotted graph theory replaces the circle with a graph.
Knotting of certain graphs has been well-studied.
That includes the theta-graph consisting of two vertices
connected by three edges---a copy of the letter $\theta$.
A \textbf{theta-curve} is an embedding of the theta-graph in Euclidean $3$-space or in the $3$-sphere.
Connected sum of knots yields the notion of a prime knot.
Similarly, one may sum two theta-curves at a vertex and one may sum a theta-curve and a knot along an edge---precise definitions
are given below.
Those two operations yield the notion of a prime theta-curve.
The work herein was prompted by the question:
\emph{how can one add an arc to a prime knot to obtain a prime theta-curve?}

Simple examples show that care must be taken.
In Figure~\ref{fig:trefoilarc} (left), the added arc is topologically parallel to an arc in the knot $K$.
\begin{figure}[htbp!]
    \centerline{\includegraphics[scale=1.0]{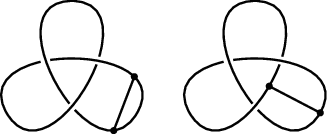}}
    \caption{Trefoil knot $K$ union an inessential arc (left) and $K$ union an essential arc (right).}
\label{fig:trefoilarc}
\end{figure}
The resulting theta-curve is a sum of the unknotted theta-curve and the knot $K$---see also Figure~\ref{fig:sum2} below---and thus is not prime.
On the other hand, adding the arc as in Figure~\ref{fig:trefoilarc} (right)
does produce a prime theta-curve by Theorem~\ref{thm:prime} below.
Thus, the choice of arc plays a decisive role in determining primeness.

Torus knots---by virtue of their definition as lying on a standard torus $T$---are particularly
amenable to study and enjoy several nice properties including: they are prime and invertible (see Figure~\ref{fig:torus} and the discussion below it).
Thus, it is natural to consider theta-curves that lie on a standard torus $T$, which we call \textbf{torus theta-curves}.
A first step towards our motivating question is to consider a torus theta-curve that is
the union of a nontrivial torus knot $K$ and an arc on $T$ that is essential in the annulus
obtained by cutting $T$ along $K$ as in Figure~\ref{fig:torustc}.
\begin{figure}[htbp!]
    \centerline{\includegraphics[scale=1.0]{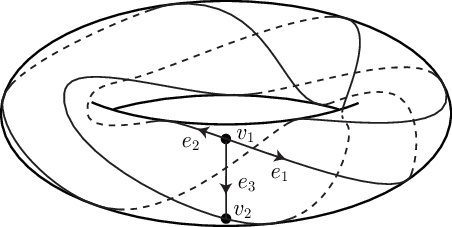}}
    \caption{Theta-curve $\theta(3,5)$ equal to the torus knot $t(3,5)$ union the arc $e_3$.}
\label{fig:torustc}
\end{figure}

We compute the constituent knots of these torus theta-curves
and use this information to obtain a complete classification of torus theta-curves
up to both ambient isotopy and homeomorphism allowing reflections.
In particular, we show that Kinoshita’s theta-curve does not lie on a standard torus.

The key feature underlying these results is the structure imposed by the torus.
Cutting the torus along a torus knot yields an annulus,
and essential arcs in this annulus have a simple and controlled topology.
This control governs both the primeness result and the classification,
allowing the geometry of the torus to constrain the possible embeddings of theta-curves.

We include examples as well as proofs of some general facts on theta-curves for their further study.
Consecutive Fibonacci numbers give rise to a structured infinite sequence of prime torus theta-curves.
Using classical Fibonacci identities, we prove that their constituent knots are
torus knots determined by consecutive Fibonacci numbers---see Examples~\ref{ex:ck}.

As the theta-graph is one of the simplest graphs after the circle, its embeddings in $3$-space have been well-studied
and arise in various contexts.
Kinoshita~\cite{kinoshita58,kinoshita72} discovered the first knotted theta-curve with all three constituent knots unknotted.
Wolcott~\cite{wolcott} studied and generalized Kinoshita's theta-curve and studied the sum of two theta-curves at a vertex.
Kauffman, Simon, Wolcott, and Zhao~\cite{kauffman} studied surfaces associated to theta-curves.
Theta-curves arose naturally in Scharlemann's work~\cite{scharlemann} on knots with both tunnel number one and genus one.
Moriuchi~\cite{moriuchi} enumerated theta-curves with up to seven crossings.
Matveev and Turaev~\cite{matveevturaev} and Turaev~\cite{turaev} introduced and studied knotoids which are intimately related to theta-curves.
The first author and Metcalf-Burton~\cite{cmb} proved a folklore conjecture of Thurston on primeness of theta-curves and double branched covers.
Baker, Buck, Moore, O'Donnol, and Taylor~\cite{taylor}
studied primeness of unknotting number one theta-curves and discussed the appearance of theta-curves in DNA replication.

The results of this paper show that, in contrast to the complexity of general theta-curves,
those lying on a torus admit a simple structural description in terms of torus knot data,
which makes their classification tractable.

This paper is organized as follows.
Section~\ref{sec:conv} recalls background material and fixes definitions.
Section~\ref{sec:prime} proves our first main result, Theorem~\ref{thm:prime}, on primeness of certain torus theta-curves.
Section~\ref{sec:ck} computes the constituent knots of those torus theta-curves, presents examples, and classifies those torus theta-curves.
Section~\ref{sec:cttcs} proves our second main result, Theorem~\ref{thm:classttc}, which classifies all torus theta-curves
up to homeomorphism of $S^3$---allowing reflections---and up to ambient isotopy.
While our focus is the classification of unoriented and unlabeled theta-curves in $S^3$, orientations turn out to play a key role.
For that reason, care with orientations is taken throughout.

\section{Background and definitions}\label{sec:conv}

We work in the smooth category $\textsc{diff}$, although it's convenient at times to utilize
the piecewise linear \textsc{pl} and topological (locally flat) \textsc{top} categories.
Those three categories are equivalent in our relevant dimensions $\leq3$.
As usual, $D^n$ denotes the closed unit disk in $\R^n$.
The boundary of $D^n$ is the $(n-1)$-sphere denoted $\partial{D^n}=S^{n-1}$.
We fix an orientation on $S^3$.
Each embedded $3$-disk $B\subseteq S^3$ is oriented as a codimension-zero submanifold of $S^3$,
and its boundary $2$-sphere $\partial B$ is given the boundary orientation using the outward normal first convention.
Using terminology of Hirsch~\cite[pp.~30--31]{hirsch},
a manifold $A$ embedded in a manifold $B$ is \textbf{neat} provided
$A\cap \partial{B}=\partial{A}$ and that intersection is transverse.
If $A$ has empty boundary, then \textbf{neat} means that $A$ is embedded in the manifold interior of $B$.

The closed unit interval is $\br{0,1}\subseteq\R$.
An \textbf{arc} is a copy of $\br{0,1}$.
A \textbf{simple closed curve}---denoted \textbf{scc}---is a copy of $S^1$.
In this paper, \textbf{surfaces} are compact and orientable, perhaps with boundary.
A scc $C$ in the manifold interior of a surface $\Sigma$ is \textbf{inessential} in $\Sigma$
provided $C$ bounds an embedded $2$-disk\footnote{Such a $2$-disk is unique unless $\Sigma$ is the $2$-sphere.}
or $C$ union a boundary component of $\Sigma$ bound an embedded annulus in $\Sigma$ as in Figure~\ref{fig:iness} (left).
\begin{figure}[htbp!]
    \centerline{\includegraphics[scale=1.0]{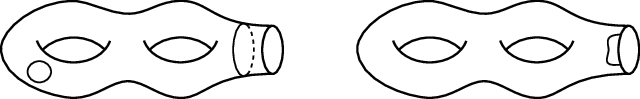}}
    \caption{Two inessential sccs in a surface $\Sigma$ with one boundary component (left) and
		an inessential arc in $\Sigma$ (right).}
\label{fig:iness}
\end{figure}
Otherwise, $C$ is \textbf{essential} in $\Sigma$.
A scc $C$ in a connected surface $\Sigma$ is \textbf{separating} provided $\Sigma-C$ is disconnected
and is \textbf{nonseparating} provided $\Sigma-C$ is connected as in Figure~\ref{fig:essscc}.
\begin{figure}[htbp!]
    \centerline{\includegraphics[scale=1.0]{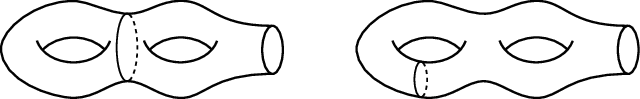}}
    \caption{Separating essential scc (left) and nonseparating essential scc (right).}
\label{fig:essscc}
\end{figure}
An arc $a$ neatly embedded in a surface $\Sigma$ is \textbf{inessential} in $\Sigma$ provided
$a$ union an arc $b$ in the boundary of $\Sigma$ bound an embedded $2$-disk in $\Sigma$ as in Figure~\ref{fig:iness} (right).
Otherwise, $a$ is \textbf{essential} in $\Sigma$.
A neatly embedded arc $a$ in a connected surface $\Sigma$ is \textbf{separating} provided $\Sigma-a$ is disconnected
and is \textbf{nonseparating} provided $\Sigma-a$ is connected as in Figure~\ref{fig:essarc}.
\begin{figure}[htbp!]
    \centerline{\includegraphics[scale=1.0]{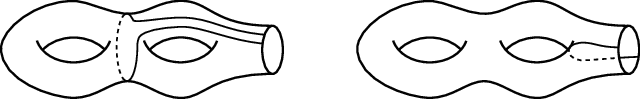}}
    \caption{Separating essential arc (left) and nonseparating essential arc (right).}
\label{fig:essarc}
\end{figure}
Each neatly embedded inessential arc is separating.

The \textbf{theta-graph} is the graph in Figure~\ref{fig:thetagraph}.
\begin{figure}[htbp!]
    \centerline{\includegraphics[scale=1.0]{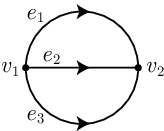}}
    \caption{The theta-graph.}
\label{fig:thetagraph}
\end{figure}
Each edge is oriented and points from $v_1$ to $v_2$.
A \textbf{theta-curve} is an embedding of the theta-graph in $S^3$.
We are interested in theta-curves up to ambient isotopy and up to homeomorphism of $S^3$
allowing reflections where we ignore graph labelings and graph orientations.
Nevertheless, those labelings and orientations are useful for calculations after which they may be forgotten.
Every theta-curve $\theta$ contains three \textbf{constituent knots} $k_1$, $k_2$, and $k_3$ where each is
obtained by ignoring one of the three edges of $\theta$.
\begin{figure}[htbp!]
    \centerline{\includegraphics[scale=1.0]{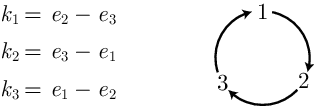}}
    \caption{Constituent knot conventions.}
\label{fig:kconv}
\end{figure}
More precisely, $k_i=e_{i+1}-e_{i+2}$ using the cyclic order in Figure~\ref{fig:kconv}.

In more detail, consider---possibly oriented---knots $K$ and $K'$ in $S^3$.
Let $K=K'$ indicate that $K$ and $K'$ are equal as sets.
Let $K\sim K'$ indicate that an ambient isotopy of $S^3$ carries $K$ to $K'$.
Let $K\approx K'$ indicate that a homeomorphism of $S^3$---perhaps reversing orientation of $S^3$---carries $K$ to $K'$.
If the knots $K$ and $K'$ are oriented, then those three equivalence relations are assumed to respect the orientations
of $K$ and $K'$ unless context dictates otherwise.
Evidently, $K=K'$ implies $K\sim K'$, and $K\sim K'$ implies $K\approx K'$.

If $K\subseteq S^3$ is an oriented knot,
then $-K$ denotes the \textbf{reverse knot} of $K$---namely, the same knot with the reverse orientation.
An oriented knot $K$ is \textbf{invertible} provided it is ambient isotopic to its reverse.
Evidently, the classifications---up to ambient isotopy of $S^3$ and
separately up to homeomorphism of $S^3$ allowing reflections---of invertible knots regarded as oriented or unoriented knots coincide.
That is relevant here since torus knots are invertible.

%Similarly, if $\theta$ is a theta-curve, then $-\theta$ denotes the \textbf{reverse theta-curve} obtained by reversing the
%roles of $v_1$ and $v_2$ and reversing the orientations on each edge.

We use those same three equivalence relations---equality, isotopy, and homeomorphism---for theta-curves,
except we always ignore graph labelings and graph orientations on theta-curves.
In particular, $\theta=\theta'$ implies $\theta\sim\theta'$, and $\theta\sim\theta'$ implies $\theta\approx \theta'$.

Let $U\subseteq S^3$ denote the unknot.
A knot $K\subseteq S^3$ is \textbf{trivial} or \textbf{unknotted} provided $K\sim U$ or equivalently $K\approx U$.
Otherwise, $K$ is \textbf{nontrivial} or \textbf{knotted}.
A theta-curve $\theta$ is \textbf{trivial} or \textbf{unknotted} provided it lies in a $2$-sphere embedded in $S^3$,
and then we denote it by $\theta_U$.
Otherwise, $\theta$ is \textbf{nontrivial} or \textbf{knotted}.
All trivial theta-curves are isotopic in $S^3$.

The constituent knots of theta-curves are important invariants.
An isotopy of $S^3$ carrying one theta-curve to another carries the unoriented constituent knots of the former to the latter.
That also holds for any homeomorphism of $S^3$.
More precisely, if $\theta\sim\theta'$, then there is a permutation $\pi$ of $\cpa{1,2,3}$ such that
$k_i\sim k'_{\pi(i)}$ as unoriented knots for $i=1,2,3$.
Similarly, if $\theta\approx\theta'$, then there is a permutation $\pi$ of $\cpa{1,2,3}$ such that
$k_i\approx k'_{\pi(i)}$ as unoriented knots for $i=1,2,3$.
Paraphrasing, theta-curves with distinct constituent knots are distinct.
Constituent knots, however, do not form a complete set of invariants for theta-curves.
For example, Kinoshita's theta-curve in Figure~\ref{fig:kinoshita} has trivial constituent knots yet is nontrivial
(see Kinoshita \cite{kinoshita58,kinoshita72} and Ozawa and Taylor~\cite{otknotted}).
\begin{figure}[htbp!]
    \centerline{\includegraphics[scale=1.0]{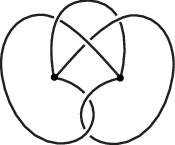}}
    \caption{Kinoshita's theta-curve.}
\label{fig:kinoshita}
\end{figure}
Since the time of Kinoshita's pioneering work in 1958, a rich collection of such theta-curves has been revealed
by Wolcott~\cite{wolcott} in 1985 and more recently in 2016 by Jang, Kronaeur, Luitel, Medici, Taylor, and Zupan~\cite{brunnian}.

Next, we recall three operations one may perform on knots and theta-curves---including the classical connected sum of knots.
In each of these operations, we delete the interiors of unknotted ball-prong pairs in separate copies of $S^3$
and glue together the resulting boundary $2$-spheres using an \emph{orientation reversing} homeomorphism.
These operations coincide with the intuitive notion of nicely splicing together two knots or graphs in a single copy
of $S^3$ and separated by an embedded $2$-sphere.
That splicing glues the two halves of $S^3$ by an orientation reversing homeomorphism of their boundaries.

A \textbf{$k$-prong} is obtained from $k$ disjoint arcs by gluing together one endpoint from each arc.
The point in a $k$-prong at which the endpoints are glued is the \textbf{prong point}.
A $2$-prong is an arc with a distinguished interior point.
A \textbf{ball-prong pair} $(B,P)$ is a copy $B$ of $D^3$ with a neatly embedded $k$-prong $P$ for some integer $k\geq2$.
The prong point of $P$ lies in the interior of $B$, and the $k$ boundary points of $P$ lie in the boundary of $B$.
A ball-prong pair $(B,P)$ is \textbf{unknotted} provided it is homeomorphic to $(D^3,R)$
where $R\subseteq D^2\subseteq D^3$ is a $k$-prong with prong point the origin and radial arcs.
In particular, an unknotted ball-arc pair is homeomorphic to $(D^3,D^1)$.
The following lemma is useful for identifying unknotted ball-prong pairs.

\begin{lemma}\label{lem:ubpp}
A ball-prong pair $(B,P)$ is unknotted if and only if there exists a neatly embedded $2$-disk
$\Delta$ in $B$ that contains $P$.
\end{lemma}

\begin{proof}
The forward implication is clear.
For the reverse implication, one builds the desired homeomorphism $h$ in stages using the following tools:
(i) the $2$- and $3$-dimensional Schoenflies theorems~\cite[pp.~68~\&~117]{moise}, \cite[Thm.~1.1]{hatcher3d}, and \cite[Ch.~III]{cerf},
and (ii) the facts that every homeomorphism of $S^1$ and $S^2$ extend to $D^2$ and $D^3$
respectively \cite[pp.~44--45]{moise}, \cite{munkres,smale}, and \cite[Thm.~3.10.11]{thurston}.
Using those tools, begin by defining a homeomorphism $h:\partial \Delta \to \partial D^2$,
extend $h$ to $P\to R$ respecting the cyclic orders of $P\subseteq\Delta$ and $R\subseteq D^2$,
extend to $\Delta\to D^2$,
extend to $\partial B\to\partial D^3$,
and finally extend to $B\to D^3$.
\end{proof}

Let $K\subseteq S^3$ and $J\subseteq S^3$ be knots in separate copies of $S^3$.
A \textbf{connected sum} of $K$ and $J$ is a knot in $S^3$ defined as follows.
Let $(B_K,a_K)$ be an unknotted ball-arc pair where $a_K$ is an arc in $K$,
and similarly define $(B_J,a_J)$.
Delete the interiors of $B_K$ and $B_J$, and then glue together the resulting ball-arc pairs along their boundaries using an orientation reversing homeomorphism matching up the two points of $K$ and of $J$ therein.
A resulting knot in $S^3$ is denoted $K\# J$.
Without additional data---such as orientations on $K$ and $J$---the connected sum of $K$ and $J$ could mean two different knots.
A knot $K$ is \textbf{prime} provided if $K\approx J_1\# J_2$,
then $J_1\approx U$ or $J_2\approx U$.
We adopt the usual convention that the unknot $U$ is not prime.
While the connected sum of knots is commutative, associative, and unital, it does not have inverses.
Namely, if $K\# J \approx U$, then $K\approx U$ and $J\approx U$.
One standard proof of that fact uses additivity of the genus of a knot~\cite[p.~93]{bz}.

Let $\theta\subseteq S^3$ be a theta-curve and let $K\subseteq S^3$ be a knot in separate copies of $S^3$.
A \textbf{$2$-connected sum} of $\theta$ and $K$ is a theta-curve in $S^3$ defined as follows.
\begin{figure}[htbp!]
    \centerline{\includegraphics[scale=1.0]{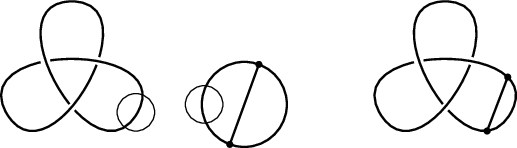}}
    \caption{A knot $K$ and a theta-curve $\theta$ with unknotted ball-arc pairs (left) and a $2$-connected sum $\theta\#_2 K$
							of $\theta$ and $K$ (right).}
\label{fig:sum2}
\end{figure}
Let $(B_{\theta},a_{\theta})$ be an unknotted ball-arc pair where $a_{\theta}$ is an arc in the interior of some edge $e_i$ of $\theta$,
and let $(B_K,a_K)$ be an unknotted ball-arc pair where $a_K$ is an arc in $K$.
Delete the interiors of $B_{\theta}$ and $B_K$, and then glue together the resulting balls along their boundaries using an orientation reversing homeomorphism matching up the two points of $\theta$ and of $K$ therein as in Figure~\ref{fig:sum2}.
A resulting theta-curve in $S^3$ is denoted $\theta\#_2 K$.
Without additional data, $\theta\#_2 K$ could mean up to six theta-curves.
Observe that $2$-connected sum does not have inverses in the following sense.
If $\theta\#_2 K\approx \theta_U$, then $\theta\approx \theta_U$ and $K\approx U$.
That's a consequence of the fact that knots don't have inverses under connected sum.

Let $\theta_1\subseteq S^3$ and $\theta_2\subseteq S^3$ be theta curves in separate copies of $S^3$.
A \textbf{$3$-connected sum} of $\theta_1$ and $\theta_2$ is a theta-curve in $S^3$ defined as follows.
\begin{figure}[htbp!]
    \centerline{\includegraphics[scale=1.0]{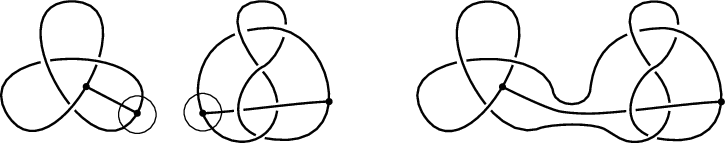}}
    \caption{Two theta-graphs $\theta_1$ and $\theta_2$ with unknotted ball-prong pairs (left) and a $3$-connected sum $\theta_1\#_3 \theta_2$
							of $\theta_1$ and $\theta_2$ (right).}
\label{fig:sum3}
\end{figure}
Let $(B_1,P_1)$ be an unknotted ball-prong pair where $P_1\subseteq \theta_1$ is a $3$-prong.
Similarly define $(B_2,P_2)$.
Delete the interiors of $B_1$ and $B_2$ and then glue together the resulting ball-prong pairs along their boundaries using an orientation reversing homeomorphism matching up the three points of $\theta_1$ and of $\theta_2$ therein as in Figure~\ref{fig:sum3}.
A resulting theta-curve in $S^3$ is denoted $\theta_1\#_3 \theta_2$.
Without additional data, $\theta_1\#_3 \theta_2$ could mean up to $24$ theta-curves.
Wolcott~\cite[$\S$4]{wolcott} proved that $\#_3$ is well-defined provided one specifies the vertices in $\theta_1$ and $\theta_2$
at which to sum and one specifies the matching of their edges.
Observe that $3$-connected sum does not have inverses in the following sense.

\begin{lemma}\label{lem:3sumnoinv}
If $\theta_1$ and $\theta_2$ are theta-curves and $\theta_1\#_3 \theta_2\approx \theta_U$,
then $\theta_1\approx \theta_U$ and $\theta_2\approx \theta_U$.
\end{lemma}

Wolcott~\cite[Thm.~4.2]{wolcott} proved Lemma~\ref{lem:3sumnoinv} using \textbf{locally unknotted}
theta-curves---meaning theta-curves in which all constituent knots are trivial---and work of Lickorish~\cite{lickorish} on prime tangles.
Motohashi~\cite[Lem.~2.1]{motohashi98} proved a generalization of Lemma~\ref{lem:3sumnoinv}.
For the convenience of the reader, we include a straightforward proof using a classical innermost circle argument.

\begin{proof}[Proof of Lemma~\ref{lem:3sumnoinv}]
By hypothesis, $\theta_U$ lies in an embedded sphere $S^2\subseteq S^3$.
As $\theta_1\#_3 \theta_2\approx \theta_U$, there is a splitting sphere $\Sigma\approx S^2$ that is transverse to $S^2$ and $\theta_U$,
and meets each edge $e_i$ at exactly one point $p_i\in\Int{e_i}$.
The spheres $\Sigma$ and $S^2$ meet in a finite disjoint union of sccs.
Evidently, one of those sccs, $C$, must meet every edge of $\theta$.
We can eliminate any other scc in $\Sigma\cap S^2$ without disturbing $\theta$ or $C$ as follows.
Let $C'$ be one of those other sccs that is innermost in $\Sigma-C$.
So, $C'=\partial \Delta$ where $\Delta$ is a $2$-disk in $\Sigma-C$ and $\Delta\cap S^2=C'$.
The scc $C'$ separates $S^2$ into two $2$-disks $D_1$ and $D_2$.
By connectedness, $\theta_U$ lies in the interior of $D_1$ or $D_2$.
Without loss of generality, $\theta_U\subseteq\Int{D_2}$.
So, the $2$-sphere $\Delta \cup D_1$ is disjoint from $\theta_U$ and separates $S^3$ into two $3$-balls.
By connectedness, one of those $3$-balls, $B$, is disjoint from $\theta_U$.
Use $B$ to push $\Delta$ slightly past $D_1$ to a parallel copy of $D_1$.
That eliminates at least $C'$ from $\Sigma\cap S^2$.
Repeat that operation finitely many times so that $\Sigma\cap S^2=C$.
Now, $C$ separates $S^2$ into two $2$-disks containing $3$-prongs of $\theta_U$.
By Lemma~\ref{lem:ubpp}, $\theta_1\approx \theta_U$ and $\theta_2\approx \theta_U$.
\end{proof}

In Section~\ref{sec:cttcs}, we classify torus theta-curves.
In Remarks~\ref{rem:class}, we use Lemma~\ref{lem:3sumnoinv} and the following lemma to observe
that a torus theta-curve is never a $3$-connected sum of nontrivial theta-curves.

\begin{lemma}\label{lem:2sumUPrimeKnot3sum}
If $\theta_U\#_2 K\approx \theta_1\#_3\theta_2$ where $K$ is a prime knot and $\theta_1$ and $\theta_2$ are theta-curves,
then $\theta_1\approx\theta_U$ or $\theta_2\approx\theta_U$.
\end{lemma}

Turaev~\cite[Lemma~5.1]{turaev} proved Lemma~\ref{lem:2sumUPrimeKnot3sum} in his seminal work on knotoids.
We include a proof for the convenience of the reader.
Our proof uses the following two lemmas.

\begin{lemma}\label{lem:chopubap}
Let $(B,a)$ be a ball-arc pair and let $\Delta\subseteq B$ be a neatly embedded $2$-disk
that meets $a$ in general position at exactly one point $p\in\Int{a}$.
Then, $\Delta$ divides $(B,a)$ into two ball-arc pairs $(B_1,a_1)$ and $(B_2,a_2)$.
Furthermore, $(B,a)$ is unknotted if and only if $(B_1,a_1)$ and $(B_2,a_2)$ are both unknotted.
\end{lemma}

\begin{proof}[Proof of Lemma~\ref{lem:chopubap}]
The $3$-dimensional Schoenflies theorem implies that $\Delta$ divides $(B,a)$ into two ball-arc pairs $(B_1,a_1)$ and $(B_2,a_2)$
where $B_1\cap B_2=\Delta$, $B_1\cup B_2 = B$, $a_1\cap a_2 =p$, and $a_1\cup a_2=a$.

For the forward implication, we are given that $(B,a)$ is unknotted.
Without loss of generality, we may assume $(B,a)=\pa{D^3,D^1}$.
In particular, $p=\Delta\cap D^1\in\Int{D^1}$.
We have standard disks $D^1\subseteq D^2\subseteq D^3$, and $\Delta$ is in general position with respect to $D^1$ and $D^2$.
We will isotop $\Delta$ in $D^3$ fixing $D^1$ pointwise and sending $\partial\Delta$ into $\partial D^3$ at all times.
By an isotopy with support near $\partial D^3$, we arrange for $\partial\Delta$ to be a great circle on $\partial D^3$.
So, $\partial\Delta$ meets $D^2$ at exactly two points.
Hence, $\Delta \cap D^2$ is the disjoint union of a single arc $b$ and finitely many sccs.
The arrangement of the endpoints of $b$ and $D^1$ on $\partial D^2$ imply that $b$ meets $D^1$.
As $\Delta$ meets $D^1$ at exactly $p$, we see that $p$ lies on $b$.
An innermost circle argument allows us to further isotop $\Delta$ so that $\Delta\cap D^2=b$.
Now, $b$ divides $D^2$ into two $2$-disks, one containing $a_1$ and the other containing $a_2$.
By Lemma~\ref{lem:ubpp}, $(B_1,a_1)$ and $(B_2,a_2)$ are unknotted, as desired.

For the reverse implication, we are given that $(B_1,a_1)$ and $(B_2,a_2)$ are unknotted.
As $(B_1,a_1)$ is unknotted, there is a neatly embedded $2$-disk $D_1\subseteq B_1$ containing $a_1$.
Adjust $D_1$ in $B_1$ by an isotopy fixing $a_1$, sending $\partial D_1$ into $\partial B_1$ at all times, and with support near $\partial B_1$
so that $D_1$ meets $\Delta$ in exactly an arc $b$ neatly embedded in $\Delta$.
Similarly, adjust $D_2$ in $B_2$.
Now, $D_1\cup D_2$ is a neatly embedded $2$-disk in $B$ containing $a$.
By Lemma~\ref{lem:ubpp}, $(B,a)$ is unknotted, as desired.
\end{proof}

\begin{lemma}\label{pkb}
Let $K\subseteq S^3$ be a prime knot.
Let $(B,b)$ be an unknotted ball-arc pair such that
$B\subseteq S^3$,
$b\subseteq K$,
and $\partial B$ meets $K$ in general position.
Let $(B',b')$ be the ball-arc pair \emph{complementary} to $(B,b)$ in $S^3$, meaning $B'=S^3-\Int{B}$ and $b'=K-\Int{b}$.
Let $\Delta\subseteq B'$ be a neatly embedded $2$-disk that meets $b'$ in general position at exactly one point $p\in\Int{b'}$.
Then, $\Delta$ divides $(B',b')$ into two ball-arc pairs, exactly one of which is unknotted.
\end{lemma}

\begin{proof}[Proof of Lemma~\ref{pkb}]
By Lemma~\ref{lem:chopubap},
$\Delta$ divides $(B',b')$ into two ball-arc pairs $(B_1,b_1)$ and $(B_2,b_2)$.
If $(B_1,b_1)$ and $(B_2,b_2)$ are both knotted, then $K$ is not prime, a contradiction.
If $(B_1,b_1)$ and $(B_2,b_2)$ are both unknotted, then the reverse implication in Lemma~\ref{lem:chopubap} implies that $(B',b')$ is unknotted.
That implies $K$ is unknotted, a contradiction.
\end{proof}

\begin{proof}[Proof of Lemma~\ref{lem:2sumUPrimeKnot3sum}]
Begin with $\theta_U$ as in Figure~\ref{fig:K2sumtU} (left).
Let $\theta'=\theta_U\#_2 K$ where, without loss of generality, the $2$-connected sum is performed along $e_3$ as in Figure~\ref{fig:K2sumtU} (middle).
\begin{figure}[htbp!]
    \centerline{\includegraphics[scale=1.0]{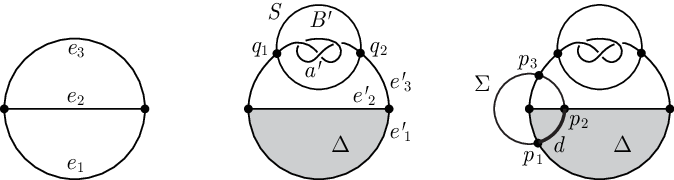}}
    \caption{Unknotted theta-graph $\theta_U$ (left), $2$-connected sum $\theta'=\theta_U\#_2 K$ where $K$ is a prime knot (middle),
		and goal in Lemma~\ref{lem:2sumUPrimeKnot3sum}.}
\label{fig:K2sumtU}
\end{figure}
That means an unknotted ball-arc pair $(B,a)$ along $e_3$ with boundary $2$-sphere $S$ is replaced with a knotted ball-arc pair
$(B',a')$ obtained from $(S^3,K)$ by removing the interior of an unknotted ball-arc pair.
The edges of $\theta'$ are $e'_1=e_1$, $e'_2=e_2$, and $e'_3$ equal to $e_3$ with $K$ tied in it.
Also depicted in Figure~\ref{fig:K2sumtU} (middle) is a $2$-disk $\Delta$ with boundary $e'_1\cup e'_2$ where $\Delta$ is disjoint from $\Int{e'_3}$ and $S$.

By hypothesis, $\theta'\approx \theta_1 \#_3 \theta_2$.
So, there is a splitting $2$-sphere $\Sigma\subseteq S^3$ that meets each arc $e'_i$
in general position at exactly one point $p_i\in\Int{e'_i}$ for $i=1,2,3$.
Notice that if $\Sigma$ and $S$ are disjoint, then the desired result holds.
To see that, replace $(B',a')$ with $(B,a)$.
Lemma~\ref{lem:3sumnoinv} implies that $\Sigma$ splits $(S^3,\theta_U)$ into two unknotted ball-prong pairs.
One of those pairs is disjoint from $B$ and is the desired unknotted ball-prong pair.

So, it remains to show that we can isotop $\Sigma$ fixing $\theta'$ setwise so that $\Sigma$ and $S$ are disjoint.
Note that $\Sigma$ meets $\Delta$ in exactly one arc $d$ and finitely many sccs $D_1,\ldots,D_m$, and
$\Sigma$ meets $S$ in finitely many sccs $C_1,\ldots,C_m$.
As $\Delta$ and $S$ are disjoint, each $C_i$ is disjoint from $d$ and from each $D_j$.
Each $C_i$ and $D_j$ has two possibilities on $\Sigma$:
\begin{enumerate}[label=(\arabic*)]
\item It bounds a $2$-disk in $\Sigma$ not containing $d$ or $p_3$.
\item It splits $\Sigma$ into two $2$-disks, one containing $d$ and one containing $p_3$.
\end{enumerate}
And, each $C_i$ has two possibilities on $S$:
\begin{enumerate}[label=(\arabic*)]
\item It bounds a $2$-disk in $S$ not containing $q_1$ or $q_2$.
\item It splits $S$ into two $2$-disks, one containing $q_1$ and one containing $q_2$.
\end{enumerate}
The following \emph{observation} will be used repeatedly: a knot in $S^3$ cannot meet a $2$-sphere in general position exactly once.
Consider an innermost scc of type (1) on $\Sigma$.
If it's a $D_j$, then an isotopy of $\Sigma$ removes it.
If it's a $C_i$, then it must also be of type (1) on $S$---by the observation applied to $e'_1\cup e'_3$---and an isotopy of $\Sigma$ removes it.
Repeat that operation until all sccs of type (1) on $\Sigma$ are removed.
By the observation applied to $e'_1\cup e'_3$, no $D_j$ remains and no $C_i$ of type (1) on $S$ remains.

Now, let $C_i$ be innermost on $\Sigma$ with respect to $p_3$---meaning $C_i$ bounds a $2$-disk $\Delta_i$ on $\Sigma$
such that $p_3\in\Int{\Delta_i}$ and $\Delta_i$ is disjoint from any other remaining $C_j$.
If $\Delta_i$ is contained in $B'$, then Lemma~\ref{pkb} implies that $\Delta_i$ splits $(B',a')$ into two ball-arc pairs, exactly one of which is unknotted.
Use that unknotted ball-arc pair to remove $C_i$.
Otherwise, $\Delta_i\subseteq S^3-\Int{B'}$.
Let $B''=S^3-\Int{B'}$ and $b''=B''\cap (e'_1\cup e'_3)$.
So, $(B'',a'')$ is an unknotted ball-arc pair and Lemma~\ref{lem:chopubap} implies that $\Delta_i$ splits $(B'',a'')$ into two unknotted ball-arc pairs.
One of those two ball-arc pairs is disjoint from $\Delta$.
Use that unknotted ball-arc pair to remove $C_i$.
\end{proof}

A theta-curve $\theta\subseteq S^3$ is \textbf{prime} provided:
(i) $\theta$ is nontrivial,
(ii) $\theta$ is not homeomorphic to a $2$-connected sum of a possibly trivial theta-curve and a nontrivial knot, and
(iii) $\theta$ is not homeomorphic to a $3$-connected sum of nontrivial theta-curves.
In other words, (ii) says that if $\theta\approx\theta'\#_2 K$ for a theta-curve $\theta'$ and a knot $K$, then $K$ is trivial.
And, (iii) says that if $\theta\approx\theta_1\#_3\theta_2$ for theta-curves $\theta_1$ and $\theta_2$, then 
$\theta_1$ or $\theta_2$ is trivial.

\begin{remark}
Some authors use an alternative definition of prime theta-curve, omitting requirement (ii) above.
For example, see Motohashi~\cite[p.~204]{motohashi98},
Matveev and Turaev~\cite{matveevturaev}, and
Turaev~\cite[p.~205]{turaev}.
\end{remark}

Let $T\subseteq S^3=\R^3\cup\cpa{\infty}$ be a standard, unknotted torus with oriented longitude $l$ and meridian $m$ as in Figure~\ref{fig:torus}.
\begin{figure}[htbp!]
    \centerline{\includegraphics[scale=1.0]{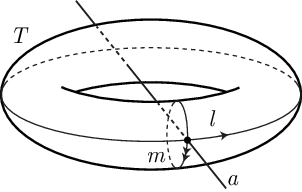}}
    \caption{Standard torus $T\subseteq S^3=\R^3\cup\cpa{\infty}$ with longitude $l$, meridian $m$,
		and the axis $a$ of the rotation $\rho$ by $\pi$ radians.}
\label{fig:torus}
\end{figure}
In particular, $T$ divides $S^3$ into two solid tori $T_1\approx S^1\times D^2$ and $T_2\approx S^1\times D^2$
such that $T_1 \cup T_2 = S^3$ and $T_1 \cap T_2 = T$.
Given relatively prime integers $p$ and $q$, let $t(p,q)\subseteq T$ denote the oriented \textbf{torus knot}
that wraps $p$ times around $T$ in the longitudinal direction and $q$ times around $T$ in the meridinal direction.
The scc $e_1-e_2$ in Figure~\ref{fig:torustc} is the torus knot $t(3,5)$.
Figure~\ref{fig:toruslm} depicts $T$ as a square with the usual identifications (left) and $t(3,5)=e_1-e_2$ in $T$ (right).
\begin{figure}[htbp!]
    \centerline{\includegraphics[scale=1.0]{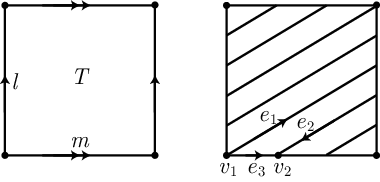}}
    \caption{Standard torus $T$ with longitude $l$ and meridian $m$ (left) and
		torus knot $t(3,5)=e_1-e_2$ and torus theta-curve $\theta(3,5)$ in $T$ (right).}
\label{fig:toruslm}
\end{figure}

Recall three involutions of $S^3=\R^3\cup\cpa{\infty}$ that are fundamental
for the classification of torus knots and for our purposes.
Each restricts to an involution of the torus $T$.
The first two are described on $\R^3$ and fix the point at infinity.
\begin{itemize}
\item $\rho:S^3\to S^3$ is the rotation of $\R^3$ by $\pi$ radians about the axis $a$ in Figure~\ref{fig:torus}.
\item $R:S^3\to S^3$ is the reflection of $\R^3$ across the horizontal plane containing the longitude $l$ in Figure~\ref{fig:torus}.
\item $\sigma:S^3\to S^3$ is the rotation of $S^3$ that swaps the core circles of $T_1$ and $T_2$,
swaps $T_1$ and $T_2$, and swaps $l$ and $m$.
\end{itemize}
The rotations $\rho$ and $\sigma$ of $S^3$ are isotopic to the identity, but the reflection $R$ is not.
Note that $\rho$ and $R$ commute, $\rho$ and $\sigma$ commute, but $R$ and $\sigma$ do not commute.

Recall some facts about torus knots and their classification.
Consider the oriented torus knot $t(p,q)$ where $p$ and $q$ are relatively prime integers.
We have:
\begin{equation}\label{symmtk}
\begin{aligned}
    -t(p,q)&=t(-p,-q) \quad &&\\
    t(p,q)&\sim -t(p,q) && \tn{by the rotation $\rho$ of $S^3$}\\
    t(p,q)&\sim t(q,p) && \tn{by the rotation $\sigma$ of $S^3$} \\
    t(p,q)&\approx t(p,-q) && \tn{by the reflection $R$ of $S^3$}
\end{aligned}
\end{equation}
The rotation $\rho$ of $S^3$ shows that torus knots are invertible, meaning they are ambient isotopic to their reverse.
Evidently, the classifications---up to ambient isotopy of $S^3$ and
separately up to homeomorphism of $S^3$ allowing reflections---of invertible knots regarded as oriented or unoriented knots coincide.
Torus knots are classified---see Burde-Zieschang~\cite[Ch.~3E]{bz} or Murasugi~\cite[Ch.~7]{mur}.

\begin{theorem}[Classification of torus knots]
Let $t(p,q)$ be a torus knot where $p$ and $q$ are relatively prime integers.
If $p$ or $q$ equals $0$ or $\pm1$, then $t(p,q)$ is the unknot.
Otherwise, $\card{p},\card{q}\geq 2$ and $t(p,q)$ is a prime knot.
Consider another such torus knot $t(p',q')$.
Then, $t(p',q')\sim t(p,q)$---as oriented knots---if and only if $\cpa{p',q'}$ equals
$\cpa{p,q}$ or $\cpa{-p,-q}$.
And, $t(p',q')\approx t(p,q)$---as oriented knots---if and only if $\cpa{p',q'}$ equals
$\cpa{p,q}$, $\cpa{-p,-q}$, $\cpa{p,-q}$, or $\cpa{-p,q}$.
The classifications of unoriented torus knots are exactly the same since torus knots are invertible.
\end{theorem}

Given a torus knot $t(p,q)$ where $p$ and $q$ are relatively prime integers and $\card{p},\card{q}\geq 2$,
we define the \textbf{torus theta-curve} $\theta(p,q)$ as follows.
In short, $\theta(p,q)$ is $t(p,q)$ union a little arc $e_3$ in $T$
as in Figures~\ref{fig:torustc}, \ref{fig:toruslm}, and \ref{fig:torusuc}.
More precisely, consider first the case where $p>0$ and $q>0$.
We choose a basepoint, denoted $v_1$, in $t(p,q)$ where $l$ and $m$ meet.
Follow $m$ from $v_1$ (using the orientation of $m$) until $t(p,q)$ is first encountered;
call that point $v_2$.
The oriented arc just described in $m$ from $v_1$ to $v_2$ is denoted $e_3$.
The oriented arc in $t(p,q)$ out of $v_1$ and ending at $v_2$ is denoted $e_1$.
The remaining arc in $t(p,q)$ out of $v_1$ and ending at $v_2$ is denoted $e_2$.
That completes our definition of $\theta(p,q)$ in case $p>0$ and $q>0$.
\begin{figure}[htbp!]
    \centerline{\includegraphics[scale=1.0]{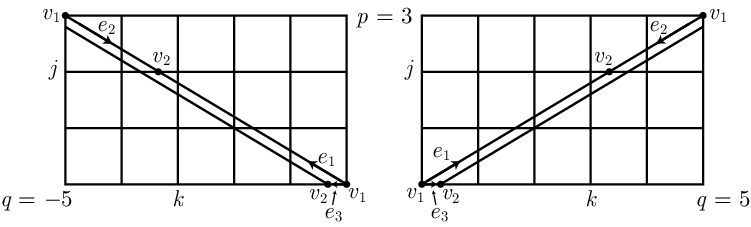}}
    \caption{Portions of the universal cover $\R^2$ of $T$ containing lifts of the arcs of torus theta-curves
		$\theta(3,-5)$ (left) and $\theta(3,5)$ (right).}
\label{fig:torusuc}
\end{figure}
Figure~\ref{fig:torusuc} (right) depicts lifts of the arcs of $\theta(p,q)$ to the universal cover $\R^2$ of $T$.
For notational convenience, lifts of $v_1$, $v_2$, $e_1$, $e_2$, and $e_3$ are denoted by the same symbols.
Now, we define torus theta-curves for arguments of other sign combinations using the involutions $\rho$ and $R$ which commute.
Using the same relatively prime integers $p\geq2$ and $q\geq2$, we define:
\begin{equation}\label{deftc}
\begin{aligned}
	&\theta(p,-q)&&=R \pa{\theta\pa{p,q}}\\
	&\theta(-p,-q)&&=\rho \pa{\theta\pa{p,q}}\\
	&\theta(-p,q)&&=R \pa{\theta\pa{-p,-q}}=\rho\pa{\theta\pa{p,-q}}
\end{aligned}
\end{equation}
Figure~\ref{fig:torustcdef} depicts local pictures of the definitions~\eqref{deftc} in $T$ near $v_1$.
\begin{figure}[htbp!]
    \centerline{\includegraphics[scale=1.0]{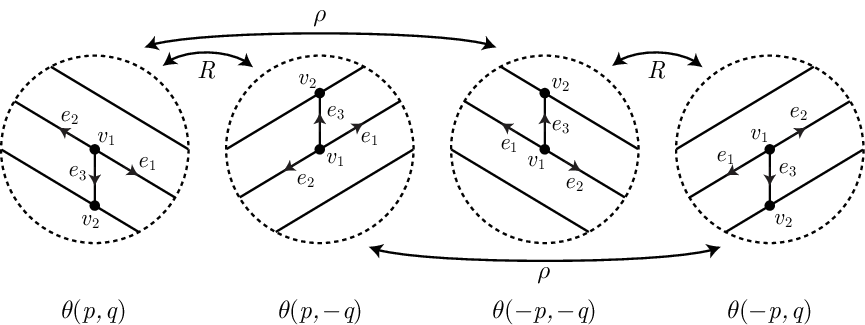}}
    \caption{Close-up view of $\theta(p,q)$ in $T$ near $v_1$ in case $p>0$ and $q>0$ (left) along with the images
		$\theta(p,-q)$, $\theta(-p,-q)$, and $\theta(-p,q)$ of $\theta(p,q)$ under the involutions $\rho$ and $R$.}
\label{fig:torustcdef}
\end{figure}

Thus, we have the following relations between theta-curves where $a$ and $b$ are relatively prime integers and $\card{a},\card{b}\geq 2$.
\begin{equation}\label{symmtcc}
\begin{aligned}
    \theta(a,b)&\sim \theta(-a,-b) && \tn{by the rotation $\rho$ of $S^3$}\\
    \theta(a,b)&\sim \theta(b,a) && \tn{by the rotation $\sigma$ of $S^3$}\\
    \theta(a,b)&\approx \theta(a,-b) && \tn{by the reflection $R$ of $S^3$} \\
    \theta(a,b)&\approx \theta(-a,b) && \tn{by the composition $\rho\circ R=R\circ\rho$}
\end{aligned}
\end{equation}

In general, sccs in $T$ are classified up to isotopy of $T$ by their intersection numbers with $l$ and $m$.
Let $H_1(T)$ denote the dimension one integer homology group of $T$.
Equip $H_1(T)\cong \Z^2$ with the ordered basis $\br{l,m}$ so that $l$ corresponds to $(1,0)\in\Z^2$ and $m$ corresponds to $(0,1)\in\Z^2$. 
In $H_1(T)$, we have $t(p,q)=e_1-e_2$.
Recall the definitions of the constituent knots $k_1$, $k_2$, and $k_3$ of $\theta(p,q)$ in Figure~\ref{fig:kconv}.
In $H_1(T)$, we have:
\begin{align*}
k_1+k_2&=-k_3\\
k_2+k_3&=-k_1\\
k_3+k_1&=-k_2
\end{align*}
Each constituent knot $k_i$ of $\theta(p,q)$ is a torus knot since $\theta(p,q)$ lies in the torus $T$.
Cutting $T$ along $t(p,q)$ yields an \textbf{annulus} $A$.
The arc $e_3$ is neatly embedded and essential in $A$ ($e_3$ has one boundary point in each boundary circle of $A$).
The algebraic intersection number---see Stillwell~\cite[p.~209]{stillwell}---of two torus knots $t=t(p,q)$ and $t'=t(p',q')$ is:
\[
I(t,t')=\det
\begin{bmatrix}
  p & p'\\
  q & q'
\end{bmatrix}
\]
Note that $I(l,m)=1$ and $I(m,l)=-1$.

We close this section with a lemma that will be useful for proving the theta-curves $\theta(p,q)$ are prime.
Recall that $T$ divides $S^3$ into two solid tori $T_1$ and $T_2$.

\begin{lemma}\label{lem:nd}
Let $\theta=\theta(p,q)$ be a torus theta-curve where $p$ and $q$ are relatively prime integers and $\card{p},\card{q}\geq 2$.
There does not exist a $2$-disk $\Delta$ neatly embedded in $T_1$ or $T_2$ such that $C=\partial{\Delta}$
is in general position with respect to $\theta$ and meets $\theta$ in exactly one point.
\end{lemma}

\begin{proof}
Suppose, by way of contradiction, that there exists such a $2$-disk $\Delta$.
If $C$ meets $e_1$ or $e_2$, then $C$ meets $t(p,q)$ once.
That implies $C$ is essential in $T$.
As $C=\partial{\Delta}$ and $\Delta$ is neatly embedded in $T_1$ or $T_2$,
$C$ is isotopic in $T$ to $l$ or $m$.
That is a contradiction since $t(p,q)$ has algebraic intersection number $-q$ with $l$ and $p$ with $m$.
So, assume $C$ meets $e_3$ once.
As $e_3$ is essential in the annulus $A$, $C$ is topologically parallel to the boundary circles of $A$.
Thus, $\Delta$ union an annulus in $A$ is a $2$-disk embedded in $S^3$ with boundary $t(p,q)$.
That is a contradiction since $t(p,q)$ is knotted in $S^3$.
\end{proof}

\section{Prime torus theta-curves}\label{sec:prime}

In this section, we prove the following theorem.

\begin{theorem}\label{thm:prime}
If $p$ and $q$ are relatively prime integers and $\card{p},\card{q}\geq 2$, then the torus theta-curve $\theta = \theta(p,q)$ is prime. 
\end{theorem}

\begin{proof}
The constituent knot $k_3 = t(p,q)$ is a non-trivial torus knot, so $\theta$ is knotted. 

Suppose, by way of contradiction, that $\theta \approx \theta' \#_2 K$ where $\theta' \subseteq S^3$ is a possibly unknotted theta-curve and $K \subseteq S^3$ is a nontrivial knot. Then, there is a 2-sphere $\Sigma \subseteq S^3$ that is transverse to $T$ and $\theta$ such that $\Sigma$ meets exactly one arc $e_i$ of $\theta$ in exactly two points $p_1, p_2 \in \Int{e_i}$. Further $\Sigma$ separates $S^3$ into two 3-balls $B_1$ and $B_2$. Without loss of generality, $B_1$ meets $\theta$ in a subarc $\alpha$ of $e_i$, $(B_1, \alpha)$ is a knotted ball-arc pair, and $v_1$ and $v_2$ lie in $\Int{B_2}$. We have $\Sigma \cap T$ is finitely many pairwise disjoint sccs.
Each such scc either meets $\theta$ or is disjoint from $\theta$. Each of the latter type is inessential in $T - \theta \approx \Int{D^2}$. We will eliminate all of the latter type by ambient isotopy of $\Sigma$ leaving $\theta$ unmoved. 

Let C be a scc of the latter type that is innermost in $T - \theta$.
So, $C = \partial \Delta$ where $\Delta$ is a 2-disk in $T - \theta$ and
$\Delta \cap \Sigma = C$.
As $\Sigma$ separates $S^3$ into $B_1$ and $B_2$, $\Delta$ lies in one of those 3-balls denoted $B_j$.
The scc $C \subseteq \Sigma$ separates $\Sigma$ into two 2-disks, $D_1$ and $D_2$. 
By connectedness, $p_1$ and $p_2$ both lie in $D_1$ or $D_2$. Without loss of generality, let $p_1$ and $p_2$ both lie in $D_1$. So $\Delta \cup D_2$ is a 2-sphere in $S^3$ that bounds a 3-disk $E$ that is disjoint from $\theta$. Use $E$ to push $D_2$ slightly past $\Delta$ to a parallel copy of $\Delta$. That eliminates at least $C$ from $\Sigma \cap T$ as desired. Repeat that operation finitely many times to eliminate all sccs in $\Sigma \cap T$ that do not meet $\theta$. 

Now, each scc in $\Sigma \cap T$ meets $\theta$.
As $\Sigma \cap \theta = \{p_1, p_2\}$ we see that $\Sigma \cap T$ consists of one or two sccs.
Suppose, by way of contradiction, that $\Sigma \cap T$ consists of two sccs, $C_1$ and $C_2$ where $p_1 \in C_1$ and $p_2 \in C_2$.
Then in $\Sigma$ we obtain two different 2-disks that contradict Lemma~\ref{lem:nd}.
Thus, $\Sigma \cap T$ is one scc, $C$, that meets $\theta$ in exactly $p_1, p_2 \in e_i$. There are three cases to consider. 
    
Case 1. $C$ meets $e_1$ twice as in Figure~\ref{fig:annulus21} (left). 
\begin{figure}[htbp!]
\centerline{\includegraphics[scale=1.0]{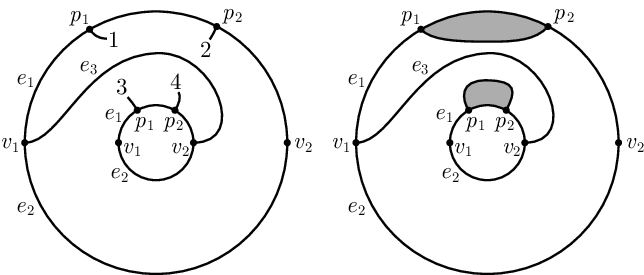}}
 \caption{The annulus $A$ in case $C$ meets $e_1$ twice (left) and the disks bounded by arcs of $C$ and $e_3$ (right).}
\label{fig:annulus21}
\end{figure}
In that figure, endpoints $1, 2, 3$, and $4$ of four arcs are indicated.
Those arcs are subarcs of $C$
that must connect in $A$ to form $C$ without meeting $\theta$ at any points other than $p_1$ and $p_2$.
If $1$ connects to $2$, then $3$ must connect to $4$.
Then the two arcs of $C$ in the annulus bound two 2-disks along arcs of $e_1$ as in Figure~\ref{fig:annulus21} (right).
Those two 2-disks form a neatly embedded 2-disk, $\Delta$, contained in $B_1$ such that $\Delta$ contains $\alpha$. That contradicts Lemma~\ref{lem:ubpp}, since $(B_1, \alpha)$ is a knotted ball-arc pair. If $1$ connects to $3$, then $C$ does not contain $p_1$, a contradiction. If $1$ connects to $4$, then $2$ and $3$ are prevented from connecting, a contradiction. 

Case 2. $C$ meets $e_2$ twice. This case is almost identical to Case 1. 

Case 3. $C$ meets $e_3$ as in Figure~\ref{fig:annulus22} (left).
\begin{figure}[htbp!]
\centerline{\includegraphics[scale=1.0]{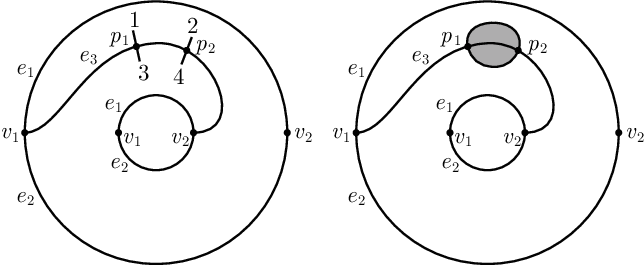}}
\caption{The annulus $A$ in case $C$ meets $e_3$ twice (left) and the disk bounded by $C$ (right).}
\label{fig:annulus22}
\end{figure}
If $1$ connects to $2$, then $3$ must connect to $4$.
Then, $C$ bounds a 2-disk, $\Delta$, in the annulus $A$ as in Figure~\ref{fig:annulus22} (right) that is neatly embedded in $B_1$ and contains $\alpha$. That contradicts Lemma~\ref{lem:ubpp}, since $(B_1, \alpha)$ is a knotted ball-arc pair. If $1$ connects to $3$, then $C$ does not contain $p_2$, a contradiction. If $1$ connects to $4$, then $2$ and $3$ are prevented from connecting, a contradiction.

Each of the three cases yielded a contradiction.
Hence, $\theta \not\approx \theta' \#_2 K$.

Suppose, by way of contradiction, that $\theta \approx \theta_1 \#_3 \theta_2$ where $\theta_1$ and $\theta_2$ are knotted theta-curves. 
Then, there is a 2-sphere $\Sigma \subseteq S^3$ that is transverse to $T$ and $\theta$ such that $\Sigma$ meets each $e_i$ at exactly one point $p_i \in \Int{e_i}$. Further, $\Sigma$ separates $S^3$ into two 3-balls $B_1$ and $B_2$ where $v_1 \in \Int{B_1}$ and $v_2 \in \Int{B_2}$, and $(B_1, B_1 \cap \theta)$ and $(B_2, B_2 \cap \theta)$ are knotted ball-prong pairs. We have $\Sigma \cap T$ is finitely many pairwise disjoint sccs. Each such scc either meets $\theta$ or is disjoint from $\theta$. Each of the latter type is inessential in $T - \theta \approx \Int{D^2}$. We will eliminate all of the latter type by ambient isotopy of $\Sigma$ leaving $\theta$ unmoved. 

Let C be a scc of the latter type that is innermost in $T - \theta$.
So, $C = \partial \Delta$ where $\Delta$ is a 2-disk in $T - \theta$ and $\Delta \cap \Sigma = C$.
As $\Sigma$ separates $S^3$ into $B_1$ and $B_2$, $\Delta$ lies in one of those 3-balls denoted $B_j$.
The scc $C \subseteq \Sigma$ separates $\Sigma$ into two 2-disks, $D_1$ and $D_2$. 
By connectedness, $p_1$, $p_2$, and $p_3$ all lie in $D_1$ or $D_2$.
Without loss of generality, assume that they all lie in $D_1$.
So $\Delta \cup D_2$ is a 2-sphere in $S^3$ that bounds a 3-disk $E$ that is disjoint from $\theta$.
Use $E$ to push $D_2$ slightly past $\Delta$ to a parallel copy of $\Delta$.
That eliminates at least $C$ from $\Sigma \cap T$ as desired.
Repeat that operation finitely many times to eliminate all sccs in $\Sigma \cap T$ that do not meet $\theta$. 

Now, each scc in $\Sigma\cap T$ meets $\theta$.
As $\Sigma\cap\theta=\cpa{p_1,p_2,p_3}$, we see that $\Sigma\cap T$ consists of one, two, or three sccs,
and those topological possibilities are depicted in Figure~\ref{fig:spherecases}.
\begin{figure}[htbp!]
    \centerline{\includegraphics[scale=1.0]{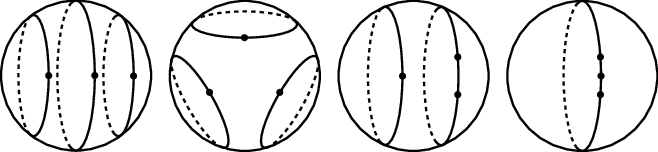}}
    \caption{The sphere $\Sigma$ in cases where $\Sigma \cap T$ consists of three, two, or one scc.}
\label{fig:spherecases}
\end{figure}
The first three possibilities each contain a $2$-disk that contradicts Lemma~\ref{lem:nd}.
Thus, $\Sigma\cap T$ is one scc, $C$, that meets each $e_i$ at exactly one point $p_i$ for $i=1,2,3$ as in Figure~\ref{fig:annulus31} (left).
\begin{figure}[htbp!]
    \centerline{\includegraphics[scale=1.0]{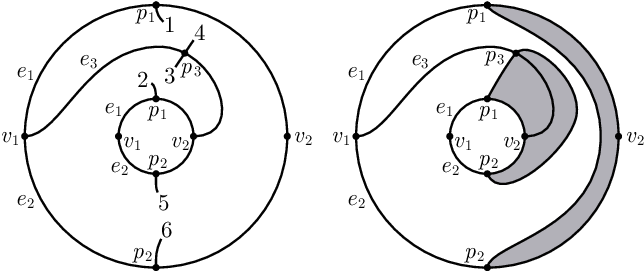}}
    \caption{The annulus $A$ where $C$ meets each arc $e_i$ once (left) and $2$-disks bounded by arcs of $C$, $e_1$, and $e_2$ (right).}
\label{fig:annulus31}
\end{figure}
In that figure, endpoints $1,2,\ldots,6$ of five arcs are indicated.
Those arcs are subarcs of $C$ that must connect in $A$ to form $C$
without meeting $\theta$ at any points other than $p_1$, $p_2$, and $p_3$.
Notice that $1$ cannot connect to $2$ since otherwise $p_2$ and $p_3$ do not lie in $C$,
$1$ cannot connect to $3$ since otherwise $6$ is closed off, and
$1$ cannot connect to $5$ since otherwise $4$ is closed off.
There are two remaining cases to consider.

Case 1. $1$ connects to $6$. Then, $4$ cannot connect to $2$ since otherwise $5$ is closed off,
and $4$ cannot connect to $3$ since $p_1,p_2,p_3$ all lie in $C$.
So, $4$ connects to $5$, and $2$ connects to $3$.
Now, the two arcs of $C$ in the annulus bound two $2$-disk along arcs of $e_1\cup e_2$ as in Figure~\ref{fig:annulus31} (right).
Those two $2$-disks form a neatly embedded $2$-disk, $\Delta$ contained in $B_1$ and containing the prong $B_1\cap\theta$.
That contradicts Lemma~\ref{lem:ubpp} since $(B_1,B_1\cap\theta)$ is a knotted ball-prong pair.

Case 2. $1$ connects to $4$. Then, $3$ cannot connect to $2$ since $p_1,p_2,p_3$ must all lie in $C$,
and $3$ cannot connect to $5$ since otherwise $2$ is closed off.
So, $3$ connects to $6$, and $2$ connects to $5$.
As in the previous case, that yields two $2$-disks in the annulus whose union $\Delta$ is
a neatly embedded $2$-disk in $B_1$ containing the prong $B_1\cap\theta$.
That contradicts Lemma~\ref{lem:ubpp} since $(B_1,B_1\cap\theta)$ is a knotted ball-prong pair.

Both cases yielded a contradiction.
Hence, $\theta \not\approx \theta_1 \#_3 \theta_2$.
Therefore, $\theta$ is a prime theta-curve.
\end{proof}

\section{Constituent knots}\label{sec:ck}

We determine the constituent knots of each torus theta-curve $\theta(p,q)$ and compute examples.
As $\theta(p,q)$ lies in the torus $T$, each constituent knot $k_i$ of $\theta(p,q)$ is a torus knot---perhaps the unknot.
For each positive integer $n$, let $\Z_n=\Z/n\Z$ denote the ring of integers modulo $n$.
Let $\br{k}_n\in\Z_n$ denote the element in $\Z_n$ represented by the integer $k$.
Recall that $\br{k}_n\in\Z_n$ has a multiplicative inverse if and only if $k$ and $n$ are relatively prime.
In that case, the multiplicative inverse of $\br{k}_n$ is denoted $\br{k}_n^{-1}$ and $\br{-k}_n^{-1}=-\br{k}_n^{-1}$.

\begin{proposition}\label{prop:ck}
Let $p$ and $q$ be relatively prime integers where $p\geq2$ and $q\geq2$.
Define the integers $j$ and $k$ by:
\begin{alignat*}{3}
\br{j}_p &= \br{q}_p^{-1} && \in \Z_p \quad && \tn{where } 1\leq j\leq p-1\\
\br{k}_q &=\br{-p}_q^{-1} && \in \Z_q \quad && \tn{where } 1\leq k\leq q-1
\end{alignat*}
Then, the constituent knots of $\theta(p,q)$ are:
\begin{alignat*}{2}
k_3 &= e_1-e_2 &&= \phantom{-}t(p,q)\\
k_2 &= e_3-e_1 &&= -t(j,k)\\
k_1 &= e_2-e_3 &&= -t(p-j,q-k)
\end{alignat*}
In particular, the set of unoriented constituent knots of $\theta(p,q)$ is:
\[
\cpa{\ t(p,q),\ t(j,k),\ t(p-j,q-k) \ }
\]
where $1\leq p-j\leq p-1$ and $1\leq q-k\leq q-1$.

Furthermore, the set of unoriented constituent knots of $\theta(p,-q)$ is:
\[
\cpa{\ t(p,-q),\ t(j,-k),\ t(p-j,k-q) \ }
\]
the set of unoriented constituent knots of $\theta(-p,-q)$ is:
\[
\cpa{\ t(-p,-q),\ t(-j,-k),\ t(j-p,k-q) \ }
\]
and the set of unoriented constituent knots of $\theta(-p,q)$ is:
\[
\cpa{\ t(-p,q),\ t(-j,k),\ t(j-p,q-k) \ }
\]
where $1-p\leq -j\leq -1$, $1-q\leq -k \leq -1$, $1-p\leq j-p\leq -1$, and $1-q\leq k-q\leq -1$.
\end{proposition}

\begin{proof}
Lift $t(p,q)$ to the universal cover $\R^2$ of $T$ as in Figure~\ref{fig:torusuc} (right).
The lift of $t(p,q)$ is the graph of $f(x)=\frac{p}{q}x$.
Thus, we seek the integer $1\leq k \leq q-1$ such that $\frac{p}{q}\pa{k+\frac{1}{p}}=j$ for some integer $1\leq j\leq p-1$.
Equivalently, we seek the integer $1\leq k \leq q-1$ such that:
\[
1=jq-kp
\]
for some integer $1\leq j\leq p-1$.
The unique solutions to the latter are those claimed in the proposition.
That determines the constituent knots of $\theta(p,q)$.

The remaining claims now follow by applying the reflection $R$ and the rotation $\rho$ to $\theta(p,q)$ and its constituent knots
(recall definition~\eqref{deftc} and relations~\eqref{symmtk}).
\end{proof}

\begin{corollary}\label{ttcinv}
Consider torus theta-curves $\theta(p,q)$ where $p$ and $q$ are relatively prime integers and $\card{p},\card{q}\geq 2$.
Among those theta-curves, $\cpa{\card{p},\card{q}}$ is a homeomorphism invariant of $\theta(p,q)$,
and whether the signs of $p$ and $q$ are the same or are opposite is an isotopy invariant of $\theta(p,q)$.
\end{corollary}

\begin{proof}
Let $\theta(p,q)$ be such a theta-curve.
The idea is that $\theta(p,q)$ contains a unique ``largest'' constituent knot from which we can recover $\cpa{\card{p},\card{q}}$.
In more detail, let $C(p,q)$ be the set of pairs $(a,b)$ of integers such that $t(a,b)\approx k_i$ for some \emph{nontrivial} constituent knot $k_i$ of $\theta(p,q)$.
We exclude trivial constituent knots since $t(\pm1,n)$ is trivial for each integer $n$ and we seek a ``largest'' element of $C(p,q)$.
By Proposition~\ref{prop:ck} and the classification of torus knots, there are exactly two pairs $(a,b)$ in $C(p,q)$ for which $a+b$ is maximal,
namely $\pa{\card{p},\card{q}}$ and $\pa{\card{q},\card{p}}$.
If $\theta(p',q')$ is another such a theta-curve and $\theta(p',q')\approx\theta(p,q)$,
then $C(p',q')=C(p,q)$ and $\cpa{\card{p'},\card{q'}}=\cpa{\card{p},\card{q}}$, as desired.

Similarly, let $D(p,q)$ be the set of pairs $(a,b)$ of integers such that $t(a,b)\sim k_i$ for some \emph{nontrivial} constituent knot $k_i$ of $\theta(p,q)$.
By Proposition~\ref{prop:ck} and the classification of torus knots, either each pair $(a,b)$ in $D(p,q)$ has $a$ and $b$ with the same sign
or each pair $(a,b)$ in $D(p,q)$ has $a$ and $b$ with opposite signs.
If $\theta(p',q')$ is another such a theta-curve and $\theta(p',q')\sim\theta(p,q)$,
then $D(p',q')=D(p,q)$ and the sign property is identical for both sets, as desired.
\end{proof}

\begin{examples}\label{ex:ck}
\noindent
\begin{enumerate}[label=(\arabic*),leftmargin=*]\setcounter{enumi}{0}
\item \emph{Adjacent integers: $p=n$ and $q=n+1$.} Consider the torus theta-curve $\theta(n,n+1)$ where $n\geq2$.
Recall that adjacent integers are relatively prime.
As in Proposition~\ref{prop:ck}, we have $j=1$, $k=1$, and the set of unoriented constituent knots of $\theta(n,n+1)$ is:
\[
\cpa{\ t(1,1),\ t(n-1,n),\ t(n,n+1) \ }
\]
By Corollary~\ref{ttcinv}, the torus theta-curves $\theta(n,n+1)$ for $n\geq2$ are pairwise inequivalent up to homeomorphism of $S^3$.
\item \emph{Consecutive Fibonacci numbers: $p=F_n$ and $q=F_{n+1}$.} Recall the Fibonacci sequence defined by
$F_0=0$, $F_1=1$, and $F_n=F_{n-2}+F_{n-1}$ for each $n\geq2$.
The Euclidean algorithm implies that consecutive Fibonacci numbers are relatively prime.
Consider the torus theta-curve $\theta(F_n,F_{n+1})$ where $n\geq3$ (so $F_n\geq 2$).
Using Proposition~\ref{prop:ck} and the computer algebra system MAGMA,
we obtained the data in Table~\ref{table:fibonacci}.
\begin{table}[h!]\renewcommand{\arraystretch}{1.2}
\begin{center}
\begin{tabular}{c|c|c|c|c|c}
$\theta$ & $\theta(2,3)$ & $\theta(3,5)$ & $\theta(5,8)$ & $\theta(8,13)$ & $\theta(13,21)$ \\ \hline
$k_3$ & $\phantom{-}t(2,3)$ & $\phantom{-}t(3,5)$ & $\phantom{-}t(5,8)$ & $\phantom{-}t(8,13)$ & $\phantom{-}t(13,21)$ \\
$k_2$ & $-t(1,1)$ & $-t(2,3)$ & $-t(2,3)$ & $-t(5,8)$ & $-t(5,8)$ \\
$k_1$ & $-t(1,2)$ & $-t(1,2)$ & $-t(3,5)$ & $-t(3,5)$ & $-t(8,13)$
\end{tabular}
\end{center}
\vspace{3mm}
\caption{Constituent knots of torus theta-curves arising from consecutive Fibonacci numbers.}
\label{table:fibonacci}
\end{table}

It appears that the constituent knots of $\theta(F_n,F_{n+1})$ are
$k_3=t(F_n,F_{n+1})$, $-t(F_{n-1},F_n)$, and $-t(F_{n-2},F_{n-1})$
where the roles of $k_2$ and $k_1$ alternate depending on the parity of $n$.
Indeed, that is the case as we now prove.

Recall two classical Fibonacci identities:
\begin{align*}
F_{n-1}F_{n+1}-F_n F_n &=(-1)^n \quad \tn{for } n\geq 1\\
F_{n-1} F_n-F_{n-2}F_{n+1} &=(-1)^n \quad \tn{for } n\geq 2
\end{align*}
The former is Cassini's identity from 1680 and
the latter is a special case of Tagiuri's identity from 1901.
Both identities are readily proven by induction as follows.
For Cassini's identity, consider the determinant of the matrix
\begin{math}
\begin{bsmallmatrix}
  F_{n-1} & F_n\\\
  F_n & F_{n+1}
\end{bsmallmatrix}
\end{math}.
For the inductive step, add the second column to the first and then swap columns.
For Tagiuri's identity, consider the determinant of the matrix
\begin{math}
\begin{bsmallmatrix}
  F_{n-1} & F_{n-2}\\\
  F_{n+1} & F_n
\end{bsmallmatrix}
\end{math}.
For the inductive step, add the first column to the second and then swap columns.

If $n\geq4$ is even, then Cassini's identity is $F_{n-1}F_{n+1}-F_n F_n =1$.
So, $j=F_{n-1}=F_{n+1}^{-1} \in \Z_{F_n}$ and
$k=F_n=-F_n^{-1} \in \Z_{F_{n+1}}$.
By Proposition~\ref{prop:ck},
the constituent knots of $\theta(F_n,F_{n+1})$ are
$k_3=t(F_n,F_{n+1})$, $k_2=-t(F_{n-1},F_n)$, and $k_1=-t(F_{n-2},F_{n-1})$.

If $n\geq3$ is odd, then Tagiuri's identity is $F_{n-2}F_{n+1}-F_{n-1} F_n =1$.
So, $j=F_{n-2}=F_{n+1}^{-1} \in \Z_{F_n}$ and $k=F_{n-1}=-F_n^{-1} \in \Z_{F_{n+1}}$.
By Proposition~\ref{prop:ck},
the constituent knots of $\theta(F_n,F_{n+1})$ are
$k_3=t(F_n,F_{n+1})$, $k_2=-t(F_{n-2},F_{n-1})$, and $k_1=-t(F_{n-1},F_n)$.

Hence, for each $n\geq3$, the set of unoriented constituent knots of $\theta(F_n,F_{n+1})$ is:
\[
\cpa{\ t(F_{n-2},F_{n-1}),\ t(F_{n-1},F_n),\ t(F_n,F_{n+1}) \ }
\]
By Corollary~\ref{ttcinv}, the torus theta-curves
$\theta(F_n,F_{n+1})$ for $n\geq3$ are pairwise inequivalent up to homeomorphism of $S^3$.
\end{enumerate}
\end{examples}

We now classify torus theta-curves of the form $\theta(p,q)$.

\begin{theorem}\label{classtpq}
Consider a torus theta-curve $\theta(p,q)$ where $p$ and $q$ are relatively prime integers and $\card{p},\card{q}\geq 2$.
Let $\theta'(p',q')$ be another such torus theta-curve.
Then, $\theta(p',q')\approx\theta(p,q)$ if and only if
$\cpa{p',q'}$ equals $\cpa{p,q}$, $\cpa{-p,-q}$, $\cpa{p,-q}$, or $\cpa{-p,q}$.
And, $\theta(p',q')\sim\theta(p,q)$ if and only if
$\cpa{p',q'}$ equals $\cpa{p,q}$ or $\cpa{-p,-q}$.
\end{theorem}

\begin{proof}
If $\theta(p',q')\approx\theta(p,q)$, then Corollary~\ref{ttcinv} implies that $\cpa{\card{p'},\card{q'}}=\cpa{\card{p},\card{q}}$, as desired.
The converse follows from the relations~\eqref{symmtcc}.

If $\theta(p',q')\sim\theta(p,q)$, then Corollary~\ref{ttcinv} implies that $\cpa{\card{p'},\card{q'}}=\cpa{\card{p},\card{q}}$
and either $p'$ and $q'$ have the same signs and $p$ and $q$ have the same signs,
or $p'$ and $q'$ have opposite signs and $p$ and $q$ have opposite signs, as desired.
The converse follows from the relations~\eqref{symmtcc}.
\end{proof}

\section{Classification of torus theta-curves}\label{sec:cttcs}

In this section, we prove the following classification of torus theta-curves.
Recall that a \textbf{torus theta-curve} is any theta-curve that lies on the standard torus $T\subseteq S^3$.
A knot or theta-curve is \textbf{knotted} provided it is knotted in $S^3$.
Each constituent knot of a torus theta-curve is a torus knot---perhaps unknotted.

\begin{theorem}[Classification of torus theta-curves]\label{thm:classttc}
If $\theta$ is a torus theta-curve, then exactly one of the following holds:
\begin{enumerate}[label=(\arabic*),leftmargin=*]\setcounter{enumi}{0}
\item All three constituent knots of $\theta$ are unknotted, in which case $\theta$ itself is unknotted.
\item At least one constituent knot of $\theta$ is knotted and another is inessential in $T$,
in which case $\theta=\theta_U \#_2\, t(a,b)$ is a $2$-connected sum of the trivial theta-curve
$\theta_U$ and a nontrivial torus knot $t(a,b)$.
\item At least one constituent knot of $\theta$ is knotted and each is essential in $T$,
in which case $\theta$ is isotopic in $T$ to $\theta(p,q)$ for some relatively prime integers $p$ and $q$ where $\card{p},\card{q}\geq2$.
\end{enumerate}
\end{theorem}

\begin{remarks}\label{rem:class}
\noindent
\begin{enumerate}[label=(\arabic*),leftmargin=*]\setcounter{enumi}{0}
\item Prime torus theta-curves---up to isotopy on $T$---are exactly the
$\theta(p,q)$ for relatively prime integers $p$ and $q$ where $\card{p},\card{q}\geq2$.
That follows immediately from Theorems~\ref{thm:prime} and~\ref{thm:classttc}.
\item A torus theta-curve is never a $3$-connected sum of nontrivial theta-curves.
By Theorem~\ref{thm:classttc}, there are three cases to consider.
The first case follows by Lemma~\ref{lem:3sumnoinv},
the second follows by Lemma~\ref{lem:2sumUPrimeKnot3sum} since nontrivial torus knots are prime,
and the third follows by Theorem~\ref{thm:prime}.
\item Given a torus theta-curve $\theta$, it is straightforward to determine which of the three mutually exclusive possibilities
in Theorem~\ref{thm:classttc} holds for $\theta$.
For each constituent knot $k_i$ of $\theta$, compute the signed intersection numbers $I(k_i,l)$ and $I(k_i,m)$
where $l$ is a longitude and $m$ is a meridian of $T$ as in Figure~\ref{fig:torus}.
Those two intersection numbers classify $k_i$ up to isotopy on $T$ (see the end of Section~\ref{sec:conv}).
Note that $k_i$ is inessential in $T$ if and only if both of those intersection numbers are zero.
Now, apply the classification of torus knots to the constituent knots of $\theta$.
\item Theorem~\ref{thm:classttc} combines with results above to completely classify torus theta-curves up
to homeomorphism of $S^3$---allowing reflections---and up to isotopy of $S^3$.
First, all unknotted theta-curves are isotopic in $S^3$.
Second, theta-curves of the form $\theta_U \#_2 K$ are classified by the knot type of $K$, namely:
\begin{equation*}
\begin{aligned}
    \theta_U \#_2 K &\approx \theta_U \#_2 K' &&\Leftrightarrow &&K\approx K'\\
		\theta_U \#_2 K &\sim \theta_U \#_2 K' &&\Leftrightarrow &&K\sim K'
\end{aligned}
\end{equation*}
Proofs of those two facts are straightforward.
(An observation in Section~\ref{sec:conv} is helpful: $\theta\#_2 K \approx \theta_U$ if and only if $K\approx U$ and $\theta\approx\theta_U$.)
Now, apply the classification of torus knots.
Third, Theorem~\ref{classtpq} classified torus theta-curves of the form $\theta(p,q)$
where $p$ and $q$ are relatively prime and $\card{p},\card{q}\geq 2$.
In summary, torus theta-curves are completely classified up to homeomorphism and up to isotopy in $S^3$
by two properties of their constituent knots: their knot types in $S^3$ and whether they are essential in $T$.
\end{enumerate}
\end{remarks}

The remainder of this section is devoted to proving Theorem~\ref{thm:classttc}.
We will state three lemmas, prove Theorem~\ref{thm:classttc} using those lemmas,
and then prove the lemmas.
Let $\theta$ denote an arbitrary theta-curve, typically in $S^3$ but in a surface where indicated.
Let $\Sigma_g$ denote the closed, orientable surface of genus $g\geq 0$.
We state and prove the first lemma for theta-curves in $\Sigma_g$.
While we only use the lemma with $g=1$, our proof works for arbitrary genus and the
general result may be useful for further study of theta-curves.

\begin{lemma}\label{lem:2ck}
Let $\theta\subseteq \Sigma_g$ where $g\geq 0$.
If $\theta$ contains at least two constituent knots that are inessential in $\Sigma_g$,
then there is a constituent knot of $\theta$ that bounds a $2$-disk $\Delta\subseteq\Sigma_g$ containing the third arc of $\theta$.
In particular, all three constituent knots of $\theta$ are inessential in $\Sigma_g$.
\end{lemma}

\begin{lemma}\label{lem:unknot}
If all three constituent knots of a torus theta-curve $\theta$ are unknotted, then $\theta$ is unknotted.
\end{lemma}

%Lemma~\ref{lem:unknot} immediately implies the following observation.

\begin{corollary}\label{cor:notttc}
If $\theta$ is knotted but all three constituent knots of $\theta$ are unknotted, then $\theta$ cannot lie on $T$.
In particular, Kinoshita's theta-curve in Figure~\ref{fig:kinoshita} is not a torus theta-curve.
\end{corollary}

\begin{remark}
A torus theta-curve may have two unknotted constituent knots and still be knotted.
Proposition~\ref{prop:ck} implies that for each positive integer $n$, $\theta(2,2n+1)$ is an example.
In fact, those are all such examples of the form $\theta(p,q)$ where $p$ and $q$ are relatively prime integers and $2\leq p <q$.
To see that, note that in Proposition~\ref{prop:ck}, $j=1$ and $p-1=1$ yield the examples $\theta(2,2n+1)$.
And, $j=1$ and $q-k=1$ is not possible since then the key equation $1=jq-kp$ becomes $1=q-(q-1)p$.
The latter implies $q(p-1)=p-1$, a contradiction.
\end{remark}

Next, we define a collection of torus theta-curves, denoted $\theta(p,q,r)$,
where $p$ and $q$ are relatively prime integers, $p,q\geq2$, and $r$ is an integer.
They generalize torus theta-curves $\theta(p,q)$ since $\theta(p,q,0)=\theta(p,q)$.
\begin{figure}[htbp!]
    \centerline{\includegraphics[scale=1.0]{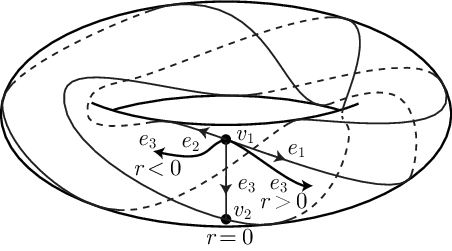}}
    \caption{Torus theta-curve $\theta(p,q,r)$ where $p$ and $q$ are relatively prime integers, $p,q\geq2$, and $r$ is an integer.
		Also indicated are various possibilities for the essential arc $e_3$ determined by the integer $r$.}
\label{fig:torustcr}
\end{figure}
Define $\theta(p,q,r)$ in the same way that $\theta(p,q)$ was defined in Section~\ref{sec:conv} except now $e_3$
winds $r\in\Z$ times around the annulus $A$ obtained by cutting $T$ along $t(p,q)$ as in Figure~\ref{fig:torustcr}.
Note that $e_3$ has a boundary point in each of the two boundary circles of $A$.
So, $e_3$ is essential in $A$.
Proposition~\ref{prop:ck} and the signed intersection numbers $I(k_i,l)$ and $I(k_i,m)$ determine the constituent knots of $\theta(p,q,r)$.
As in Proposition~\ref{prop:ck}, define the integers $j$ and $k$ by:
\begin{alignat*}{3}
\br{j}_p &= \br{q}_p^{-1} && \in \Z_p \quad && \tn{where } 1\leq j\leq p-1\\
\br{k}_q &=\br{-p}_q^{-1} && \in \Z_q \quad && \tn{where } 1\leq k\leq q-1
\end{alignat*}
The constituent knots of $\theta(p,q,r)$ are:
\begin{equation}\label{cktpqr}
\begin{alignedat}{2}
k_3 &= e_1-e_2 &&= t(p,q)\\
k_2 &= e_3-e_1 &&= t(rp-j,rq-k)\\
k_1 &= e_2-e_3 &&= t(j-(r+1)p,k-(r+1)q)
\end{alignedat}
\end{equation}

\begin{lemma}\label{lem:esscase}
If $\theta$ is a torus theta-curve, $\theta$ contains a knotted constituent knot, and each constituent knot of $\theta$
is essential in $T$, then $\theta$ is isotopic on $T$ to $\theta(p,q)$ for some
relatively prime integers $p$ and $q$ where $\card{p},\card{q}\geq2$.
\end{lemma}

We now use the three lemmas to prove Theorem~\ref{thm:classttc}.

\begin{proof}[Proof of Theorem~\ref{thm:classttc}]
If all three constituent knots of $\theta$ are unknotted, then $\theta$ is unknotted by Lemma~\ref{lem:unknot}.
Otherwise, $\theta$ contains at least one knotted constituent knot.
Without loss of generality, assume $k_3=e_1-e_2$ is knotted.
Then, $k_3=t(a,b)$ for some relatively prime integers $a$ and $b$ where $\card{a},\card{b}\geq 2$.
Let $A$ denote the annulus obtained by cutting $T$ along $k_3$.
As $k_3$ is essential in $T$, $\theta$ has a constituent knot inessential in $T$ if and only if $e_3$ is inessential in $A$.
So, if $\theta$ has a constituent knot inessential in $T$, then $\theta=\theta_U \#_2\, t(a,b)$, as desired.
Otherwise, each constituent knot of $\theta$ is essential in $T$ and the result follows by Lemma~\ref{lem:esscase}.
\end{proof}

It remains to prove the three lemmas.

\begin{proof}[Proof of Lemma~\ref{lem:2ck}]
If $g=0$, then the result follows by the $2$-dimensional Schoenflies theorem~\cite[p.~71]{moise}.
So, assume $g>0$.
Let $\pi:\wt{\Sigma}_g\to\Sigma_g$ be the universal covering map where $\wt{\Sigma}_g\approx\R^2$.
Without loss of generality, assume $k_1=e_2-e_3$ and $k_2=e_3-e_1$ are inessential in $\Sigma_g$.
Let $\Delta_1$ and $\Delta_2$ be the $2$-disks in $\Sigma$ bounded by $k_1$ and $k_2$ respectively.
If $\Delta_1$ contains $e_1$ or $\Delta_2$ contains $e_2$, then we are done.
So, assume $e_1$ does not lie in $\Delta_1$ and $e_2$ does not lie in $\Delta_2$.
The $2$-disks $\Delta_1$ and $\Delta_2$ are given by embeddings $f_1:D^2\to \Sigma_g$ and $f_2:D^2\to \Sigma_g$ respectively.
Lift $f_1$ to $\wt{f}_1:D^2\to \wt{\Sigma}_g$ so $f_1=\pi\circ \wt{f}_1$.
Note that $\wt{f}_1$ is an embedding, and the restriction of $\wt{f}_1$ to $\partial D^2$ is a lift of $k_1$ to a scc in $\wt{\Sigma}_g$.
Let $\wt{v}_1$ and $\wt{v}_2$ denote the lifts under $f_1$ of $v_1$ and $v_2$ respectively.
Lift $f_2$ to $\wt{f}_2:D^2\to \wt{\Sigma}_g$ so $f_2=\pi\circ \wt{f}_1$ and $v_1$ lifts to $\wt{v}_1$ (see Moise~\cite[p.~176]{moise}).
Note that $\wt{f}_2$ is an embedding, and the restriction of $\wt{f}_2$ to $\partial D^2$ is a lift of $k_2$ to a scc in $\wt{\Sigma}_g$.
Restricting $f_1$ and $f_2$ to arcs in the boundary of $D^2$ corresponding to $e_3$,
uniqueness of lifts implies that $f_2$ lifts $v_2$ to $\wt{v}_2$.
Thus, we may restrict $f_1$ and $f_2$ to $\partial D^2$ and paste together those restrictions to obtain
a lift $\wt{\theta}$ of $\theta$ embedded in $\wt{\Sigma}_g$.
Note that $\pi$ restricted to $\wt{\theta}$ is injective.
As $\wt{\Sigma}_g\approx\R^2$, Moise~\cite[pp.~20~\&~68]{moise} implies that exactly one constituent knot of $\wt{\theta}$ bounds a $2$-disk
$\wt{\Delta}\subseteq\wt{\Sigma}_g$ containing the third arc of $\wt{\theta}$.
As $\pi$ restricted to $\partial \wt{\Delta}$ is injective, Lemma~1.6 of Epstein~\cite{epstein} implies that
$\pi$ restricted to $\wt{\Delta}$ is injective.
Define $\Delta=\pi\pa{\wt{\Delta}}$ which is a $2$-disk embedded in $\Sigma$, bounded by a constituent knot of $\theta$,
and containing the third arc of $\theta$.
Recall that $e_1$ does not lie in $\Delta_1$, $e_2$ does not lie in $\Delta_2$,
and each inessential scc in $\Sigma_g$ (where $g>0$) bounds a unique $2$-disk in $\Sigma_g$.
Hence, $\partial \Delta = k_3=e_1-e_2$ and $\Delta$ contains $e_3$, as desired.
\end{proof}

\begin{proof}[Proof of Lemma~\ref{lem:unknot}]
If at least two constituent knots of $\theta$ are inessential in $T$, then $\theta$ lies in a $2$-disk in $T$ by Lemma~\ref{lem:2ck}
and hence is unknotted.
So, assume $\theta$ contains a constituent knot that is essential in $T$.
Recall the standard longitude $l$ and meridian $m$ of $T$ in Figure~\ref{fig:toruslm}.
Evidently, if a constituent knot of $\theta$ is isotopic in $T$ to $m$ or to $l$,
then $\theta$ is unknotted.
By the classification of torus knots and the relations~\eqref{symmtk},
it remains, without loss of generality, to consider the case $k_3=e_1-e_2=t(1,q)$ where $q$ is a positive integer.
Let $A$ denote the annulus obtained by cutting $T$ along $k_3$.
If $e_3$ is inessential in $A$, then $\theta$ is unknotted.
So, assume $e_3$ is essential in $A$.
That implies $e_3$ has a boundary point in each of the two boundary circles of $A$.
Hence, up to isotopy on $T$ fixing $k_3$ pointwise, $e_3$ is determined by an integer $r$ that records the number of times
$e_3$ winds around $A$ as in Figure~\ref{fig:torust1q}.
\begin{figure}[htbp!]
    \centerline{\includegraphics[scale=1.0]{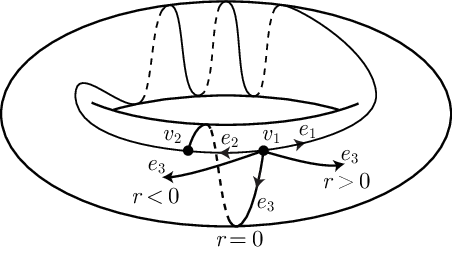}}
    \caption{Constituent knot $k_3=e_1-e_2=t(1,q)$ where $q$ is a positive integer ($q=3$ is depicted) and various
		possibilities for the essential arc $e_3$ determined by the integer $r$.}
\label{fig:torust1q}
\end{figure}
Note that for any integer $r$, $e_3$ always terminates at $v_2$ \emph{on the same side} of $k_3$ since $e_3$ is essential in $A$.
The signed intersection numbers $I(k_i,l)$ and $I(k_i,m)$ determine the constituent knots of $\theta$:
\begin{equation}\label{ck1q}
\begin{alignedat}{2}
k_3 &= e_1-e_2 &&= t(1,q)\\
k_2 &= e_3-e_1 &&= t(r-1,q(r-1)+1)\\
k_1 &= e_2-e_3 &&= t(-r,-qr-1)
\end{alignedat}
\end{equation}
In each case $q\in\Z_+$ and $r\in\Z$, either a constituent knot of $\theta$ is isotopic on $T$ to $m$ or $l$ implying
$\theta$ is unknotted, or a constituent knot of $\theta$ is knotted which contradicts the hypotheses.
In more detail, if $r=0$, then $k_1=m$.
If $r=1$, then $k_2=m$.
If $r\geq2$, then $k_1$ is knotted.
If $r=-1$ and $q=1$, then $k_1=l$.
If $r=-1$ and $q\geq2$, then $k_2$ is knotted.
If $r\leq-2$, the $k_2$ is knotted.
\end{proof}

\begin{proof}[Proof of Lemma~\ref{lem:esscase}]
By hypothesis, $\theta\subseteq T$ has a constituent knot $t(p',q')$ where $p'$ and $q'$ are relatively prime integers and $\card{p'},\card{q'}\geq2$.
The desired conclusion is invariant under reflection of $S^3$.
So, we can and do assume $p',q'\geq2$.
Without loss of generality, assume $k_3=t(p',q')$.
Let $A$ denote the annulus obtained by cutting $T$ along $k_3$.
By hypothesis, each constituent knot of $\theta$ is essential in $T$.
Thus, $e_3$ is essential in $A$ and $e_3$ has a boundary point in each of the two boundary circles of $A$.
Up to isotopy on $T$ fixing $k_3$ pointwise, $e_3$ is determined by an integer $r'$ that records the number of times
$e_3$ winds around $A$ as in Figure~\ref{cktpqr}.
So, $\theta=\theta(p',q',r')$ where $p'$ and $q'$ are relatively prime integers, $p',q'\geq 2$, and $r'$ is an integer.
If $r'=0$, then we're done.
So, assume $r'\ne 0$.

The constituent knots of $\theta(p',q',r')$ were determined in~\eqref{cktpqr}.
Among those three torus knots $t(a,b)$, there is one that is ``largest'' in the sense that $\card{a}+\card{b}$ is maximal.
Direct inspection of~\eqref{cktpqr} shows that $k_1$ is largest when $r'>0$ and $k_2$ is largest when $r'<0$.
Furthermore, direct inspection of~\eqref{cktpqr} shows that in both of those cases the largest constituent knot is knotted.
Isotop $\theta(p',q',r')$ on $T$ to $\theta(p,q,r)$ so that the largest constituent knot of $\theta(p',q',r')$ is carried to $k_3$ of $\theta(p,q,r)$.
As the largest constituent knot of $\theta(p',q',r')$ is knotted, we have that $p$ and $q$ are relatively prime integers, $p,q\geq 2$, and $r$ is an integer.
Now, let $k_i$ refer to a constituent knot of $\theta(p,q,r)$.
As $k_3$ is the largest constituent knot of $\theta(p,q,r)$,
direct inspection of~\eqref{cktpqr} implies that $r=0$, as desired.
\end{proof}


\begin{thebibliography}{99}

\bibitem[BBOMT25]{taylor}
K.~Baker, D.~Buck, A.~Moore, D.~O'Donnol, and S.~Taylor,
\emph{Primality of theta-curves with proper rational tangle unknotting number one},
Trans. Amer. Math. Soc. Ser. B \textbf{12} (2025), 276--297.

%\bibitem[Bir75]{birman}
%J.S.~Birman,
%\emph{Braids, Links, and Mapping Class Groups},
%based on lecture notes by James Cannon,
%Princeton University Press, Princeton, N. J.;
%University of Tokyo Press, Toyko, 1975.

\bibitem[BZ03]{bz}
G.~Burde and H.~Zieschang,
\emph{Knots},
De Gruyter Studies in Mathematics,
Walter de Gruyter \& Co., Berlin,
2003.

%\bibitem[CMBRS14]{borr}
%J.S.~Calcut, J.R.~Metcalf-Burton, T.J.~Richard, and L.T.~Solus,
%\emph{Borromean rays and hyperplanes},
%J. Knot Theory Ramifications \textbf{23} (2014), 46 pp.

\bibitem[CMB16]{cmb}
J.S.~Calcut and J.R.~Metcalf-Burton,
\emph{Double branched covers of theta-curves},
J. Knot Theory Ramifications \textbf{25} (2016), 9 pp.

\bibitem[Cer68]{cerf}
J.~Cerf,
\emph{Sur les diff\'eomorphismes de la sph\`ere de dimension trois $(\Gamma_{4}=0)$},
Lecture Notes in Mathematics \textbf{53},
Springer-Verlag, Berlin, 1968.

\bibitem[Eps66]{epstein}
D.B.A.~Epstein,
\emph{Curves on $2$-manifolds and isotopies},
Acta Math. \textbf{115} (1966), 83--107.
   
%\bibitem[FM12]{fm}
%B.~Farb and D.~Margalit,
%\emph{A primer on mapping class groups},
%Princeton Mathematical Series \textbf{49},
%Princeton University Press, Princeton, NJ, 2012.
	
% \bibitem[GP74]{gp}
% 	V.~Guillemin and A.~Pollack,
% 	\emph{Differential topology},
% 	Prentice--Hall, Englewood Cliffs, NJ, 1974.

\bibitem[Hat00]{hatcher3d}
A.~Hatcher,
\emph{Notes on basic 3-manifold topology},
available at
\href{http://www.math.cornell.edu/~hatcher/3M/3M.pdf}{\curl{http://www.math.cornell.edu/$\sim$hatcher/3M/3M.pdf}}, 2000.

%\bibitem[Hat02]{hatcher}
%	---------,
%	\emph{Algebraic topology},
%	Cambridge University Press, Cambridge, 2002.
%

\bibitem[Hir76]{hirsch}
M.W.~Hirsch,
\emph{Differential topology},
Springer--Verlag,
New York, 1994 (Corrected reprint of the 1976 original).

\bibitem[JKLMTZ16]{brunnian}
B.~Jang,
A.~Kronaeur,
P.~Luitel,
D.~Medici,
S.A.~Taylor, and
A.~Zupan,
\emph{New examples of Brunnian theta graphs},
Involve \textbf{9} (2016), 857--875.

\bibitem[KSWZ93]{kauffman}
L.~Kauffman, J.~Simon, K.~Wolcott, and P.~Zhao,
\emph{Invariants of theta-curves and other graphs in $3$-space},
Topology Appl. \textbf{49} (1993), 193--216.

\bibitem[Kin58]{kinoshita58}
S.~Kinoshita,
\emph{Alexander polynomials as isotopy invariants. I},
Osaka Math. J. \textbf{10} (1958), 263--271.

\bibitem[Kin72]{kinoshita72}
S.~Kinoshita,
\emph{On elementary ideals of polyhedra in the $3$-sphere},
Pacific J. Math. \textbf{42} (1972), 89--98.

\bibitem[Lic81]{lickorish}
W.B.~Lickorish,
\emph{Prime knots and tangles},
Trans. Amer. Math. Soc. \textbf{267} (1981), 321--332.

\bibitem[MT11]{matveevturaev}
S.~Matveev and V.~Turaev,
\emph{A semigroup of theta-curves in 3-manifolds},
Mosc. Math. J. \textbf{11} (2011), 805--814, 822.

\bibitem[Moi77]{moise}
E.~E.~Moise,
\emph{Geometric topology in dimensions $2$ and $3$},
Springer-Verlag, New York-Heidelberg, 1977.

\bibitem[Mor09]{moriuchi}
H.~Moriuchi,
\emph{An enumeration of theta-curves with up to seven crossings},
J. Knot Theory Ramifications \textbf{18} (2009), 167--197.

\bibitem[Mot98]{motohashi98}
T.~Motohashi,
\emph{A prime decomposition theorem for {$\theta_n$}-curves in {$S^3$}},
Topology Appl. \textbf{83} (1998), 203--211.

\bibitem[Mun60]{munkres}
J.~Munkres,
\emph{Differentiable isotopies on the $2$--sphere},
Michigan Math. J. \textbf{7} (1960), 193--197.

\bibitem[Mur08]{mur}
K.~Murasugi,
\emph{Knot theory \& its applications},
Modern Birkh\"auser Classics,
Translated from the 1993 Japanese original by Bohdan Kurpita;
Reprint of the 1996 translation,
Birkh\"auser Boston, Inc., Boston, MA, 2008.

\bibitem[OT19]{otknotted}
M.~Ozawa and S.A.~Taylor,
\emph{Two more proofs that the Kinoshita graph is knotted},
Amer. Math. Monthly \textbf{126} (2019), 352--357.

%\bibitem[Rol90]{rolfsen}
%D.~Rolfsen,
%\emph{Knots and links},
%Corrected reprint of the 1976 original,
%Publish or Perish, Houston, 1990.

\bibitem[Sch01]{scharlemann}
M.~Scharlemann,
\emph{The Goda-Teragaito conjecture: an overview},
in \emph{On Heegaard splittings and Dehn surgeries of 3-manifolds, and
   topics related to them} (Kyoto, 2001),
S\=urikaisekikenky\=usho K\=oky\=uroku,
1229 (2001), 87--102.

%\bibitem[ST91]{scharlemannthompson}
%M.~Scharlemann and A.~Thompson,
%\emph{Detecting unknotted graphs in $3$-space},
%J. Differential Geom. \textbf{34} (1991), 539--560.

\bibitem[Sma59]{smale}
S.~Smale,
\emph{Diffeomorphisms of the $2$--sphere},
Proc. Amer. Math. Soc. \textbf{10} (1959), 621--626.

\bibitem[Sti93]{stillwell}
J.~Stillwell,
\emph{Classical topology and combinatorial group theory},
Springer-Verlag, New York, 1993.

%\bibitem[Tan93]{taniyama}
%K.~Taniyama,
%\emph{Cobordism of theta curves in $S^3$},
%Math. Proc. Cambridge Philos. Soc. \textbf{113} (1993), 97--106.

\bibitem[Thu97]{thurston}
W.P.~Thurston,
\emph{Three-dimensional geometry and topology. Vol. 1},
Edited by Silvio Levy, Princeton University Press, Princeton, NJ, 1997.

\bibitem[Tur12]{turaev}
V.~Turaev,
\emph{Knotoids},
Osaka J. Math. \textbf{49} (2012), 195--223.

\bibitem[Wol87]{wolcott}
K.~Wolcott,
\emph{The knotting of theta curves and other graphs in $S^3$},
Geometry and topology (Athens, Ga., 1985),
Lecture Notes in Pure and Appl. Math., vol. 105,
Dekker, New York, 1987, 325--346.

\end{thebibliography}
\end{document}